\documentclass{amsart}
\usepackage{latexsym,amsmath,amsthm,verbatim,ifthen,amssymb,graphicx,color}
\usepackage[all]{xy}
\usepackage{color}
\usepackage{tikz}
\usepackage{abraces}

\makeatletter
\newcommand\RSloop{\@ifnextchar\bgroup\RSloopa\RSloopb}
\makeatother
\newcommand\RSloopa[1]{\bgroup\RSloop#1\relax\egroup\RSloop}
\newcommand\RSloopb[1]%
  {\ifx\relax#1%
   \else
     \ifcsname RS:#1\endcsname
       \csname RS:#1\endcsname
     \else
       \GenericError{(RS)}{RS Error: operator #1 undefined}{}{}%
     \fi
   \expandafter\RSloop
   \fi
  }
\newcommand\X{0}
\newcommand\RS[1]%
  {\begin{tikzpicture}
     [every node/.style=
       {circle,draw,fill,minimum size=1.5pt,inner sep=0pt,outer sep=0pt},
      line cap=round
     ]
   \coordinate(\X) at (0,0);
   \RSloop{#1}\relax
   \end{tikzpicture}
  }
\newcommand\RSdef[1]{\expandafter\def\csname RS:#1\endcsname}
\newlength\RSu
\RSdef{i}{\draw (\X) -- +(90:\RSu) node{};}
\RSdef{l}{\draw (\X) -- +(135:\RSu) node{};}
\RSdef{r}{\draw (\X) -- +(45:\RSu) node{};}
\RSdef{I}{\draw (\X) -- +(90:\RSu) coordinate(\X I);\edef\X{\X I}}
\RSdef{L}{\draw (\X) -- +(135:\RSu) coordinate(\X L);\edef\X{\X L}}
\RSdef{R}{\draw (\X) -- +(45:\RSu) coordinate(\X R);\edef\X{\X R}}
\SelectTips{cm}{}

\newtheorem{theorem}{Theorem}[section]

 \newtheorem{lemma}[theorem]{Lemma}
 
 \theoremstyle{definition}
 \newtheorem{definition}[theorem]{Definition}
 \theoremstyle{remark}
 \newtheorem{remark}[theorem]{Remark}
 \newtheorem{example}[theorem]{Example}
 \numberwithin{equation}{section}
\newtheorem*{Introtheorem}{\bf Theorem}

\begin{document}

\title[Explicit Quillen models for Cartesian products of $2$-cones]{Explicit Quillen models for Cartesian products of $2$-cones}

\author[U. Buijs]{Urtzi Buijs}
\address{Departamento de Algebra, Geometr\'{\i}a y Topolog\'{\i}a, Universidad
de M\'alaga, Ap. 59, 29080 M\'alaga, Spain}
\email{ubuijs@uma.es}

\author[J. Carrasquel]{Jos\'e Carrasquel}
\email{jgcarras@gmail.com}

\author[L. Vandembroucq]{Lucile Vandembroucq}
\address{Centro de Matemática, Universidade do Minho,
 Campus de Gualtar,
4710-057 Braga,
Portugal}
\email{lucile@math.uminho.pt}


\maketitle

\begin{abstract}
We give an explicit minimal Quillen model for the Cartesian product $X\times Y$ of rational $2$-cones in terms of derivations and a binary operation $\star \colon \mathbb{M}(V)\otimes \mathbb{L}(W)\to \mathbb{L}(V\oplus W\oplus s(V\otimes W))$, where $(\mathbb{L}(V), \partial)$ and $(\mathbb{L}(W), \partial)$ are Quillen minimal models for $X$ and $Y$ respectively and $\mathbb{M}$ denotes the free magma on $W$.

The model presented also allows us to explicitly describe a model for the diagonal map $\Delta \colon X\to X\times X$.
\end{abstract}

\section{introduction}

In \cite[VII.1.(2)]{Tanre}, a minimal Quillen model for the product of two spaces $X$ and $Y$ is given in terms of minimal Quillen models of $X$ and $Y$. Namely,

\begin{Introtheorem}
Let $X$ and $Y$ be two pointed topological spaces of finite type and $1$-connected, with minimal Quillen models $(\mathbb{L}(V), \partial_V)$ and $(\mathbb{L}(W), \partial_W)$ respectively. Then, the Cartesian product $X\times Y$ has a minimal Quillen model of the form
\begin{equation}\label{modelTanre}
\Psi \colon \Bigl(\mathbb{L}(V\oplus W \oplus s(V\otimes W)), D\Bigr) \longrightarrow \Bigl(\mathbb{L}(V)\times \mathbb{L}(W), \partial_V\times \partial_W\Bigr)
\end{equation}
where $\Psi (v)=v$, $\Psi (w)=w$, $\Psi (s(v\otimes w))=0$, $Dv=\partial _V (v)$, $Dw=\partial_W (w)$ and $D(s(v\otimes w))=[v, w] +\beta (v, w)$, with $v\in V$, $w\in W$ and $\beta (v, w)$ a decomposable element in the Lie algebra $\text{Ker }\Psi $.
\end{Introtheorem} 

The proof of this theorem provides a procedure to obtain the model of the form $(\ref{modelTanre})$ but it could be lengthy and tedious.

Explicit formulas for the differential of the above model have turned hard to obtain. In \cite{Lupton-Smith} it was given such a formula for the particular case in which one factor of the product $X\times Y$ is a rational co-H-space,  having in consequence a Quillen model of the form $(\mathbb{L}(V), 0)$. If both factors have models with non-zero differenials the complexity of model $(\ref{modelTanre})$ grows considerably.
\vskip.3cm

Far from being a simple task, knowing the rational homotopy of the Cartesian product is a matter of great complexity and significance.

Not in vain, ``open question'' number 15 of the 17 collected as a challenge for the rational homotopy community in \cite{FHT} is stated as follows:
\vskip .3cm
\emph{Suppose $X$, $Y$ and $Z$ are simply connected rational CW complexes with homology of finite type. Is it true that $X\times Z \simeq X \times Y$ implies $Z\simeq Y$?}
\vskip .3cm

In addition to these fundamental questions, the study of invariants such as the sectional category of a continuous map $f\colon X\to Y$ (which generalizes the Lusternik-Schnirelmann (LS) category \cite{LS} and Farber's topological complexity \cite{F}) can be stated in terms of Cartesian products.

Indeed, if $f\colon X \hookrightarrow Y$ is the inclusion of a sub CW-complex then, the sectional category of $f$ is the smallest $n$ for which the diagonal 
$$\Delta_{n+1} \colon Y \to \underbrace{Y\times \cdots \times Y}_{\text{$n+1$ times}}$$
factors up to homotopy through the fat wedge.

Traditionally, the study of these invariants has been carried out following Sullivan's approach to rational homotopy in the framework of commutative differential graded algebras \cite{FH}, as for example the product formula for rational category $\text{cat}_0 (X\times Y)=\text{cat}_0 X+ \text{cat}_0Y$ proved in \cite{FHL}.

This product formula was conjectured during a time in which particular cases were tested, with special relevance being Ganea's conjecture ($\text{cat}(X\times S^n)=\text{cat}X + 1$), which was rationally proven in \cite{H}.

These results, involving Cartesian products, are not satisfied for other invariants, such as cone-length, as was proven in \cite{D}. Since the study of cone-length is best suited to Quillen's treatment of rational homotopy, using graded differential Lie algebras (dgl) it is of particular importance to understand Cartesian products from this perspective.

\vskip .2cm
The main goal of this work is the development of an explicit differential for the Quillen minimal model of the Cartesian product $X\times Y$ of two $2$-cones, i.e., $X$ and $Y$ have minimal Quillen models $(\mathbb{L}(V), \partial_V)$ and $(\mathbb{L}(W), \partial_W)$ respectively, where $V=V_0\oplus V_1$ and $W=W_0\oplus W_1$ with $\partial _V (V_0) = \partial _W (W_0)=0$ and $\partial (V_1)\subset \mathbb{L}(V_0)$ and $\partial (W_1)\subset \mathbb{L}(W_0)$.

This is the first step in constructing a general explicit model for the product defined in terms of a cone decomposition of the factors and we present the tools and techniques necessary for this purpose. Having these explicit models is a requirement to begin the study of the sectional category in terms of Quillen models carried out in \cite{BC}.

Section 2 is dedicated to developing the necessary preliminaries and presenting the known examples. In particular, we show the case of the Cartesian product in which a factor is a co-H space described in \cite{Lupton-Smith} where the differential is given in terms of derivations.
\vskip .3cm
Section 3.1 is devoted to define the differential $D$ in 
$$L_1 = \mathbb{L}(V\oplus W \oplus s(V_0 \otimes W_0) \oplus s(V_1\otimes W_0)\oplus s(V_0 \otimes W_1)),$$
and to prove some combinatorial lemmas relating this differential with derivations in $\text{Der}(L)$, where $L=\mathbb{L}(V\oplus W \oplus s(V\otimes W))$.

In Section 3.2 we present a product:
$$\star \colon \mathbb{M}(V)\otimes \mathbb{L}(W)\to \mathbb{L}(V\oplus W \oplus s(V\otimes W)),$$ 
where $\mathbb{M}(V)$ denotes the magma on a graded vector space $V$ which will be essential to define the differential in our model.

This product has some interesting properties, as for example, that it does not respect Jacobi identity. This is precisely the reason why it is necessary to introduce the magma on $V$ instead of the free Lie algebra.

The formula presented in Section 3.3 for the differential in the Cartesian product of $2$-cones, namely
$$D( s(v\otimes w))=[v, w]-(-1)^{(|v|+1)|w|}\sigma_{w}(\partial_V v) -(-1)^{|v|}\sigma_v (\partial_W w)+(-1)^{|v|}\overline{\partial_V v}\star \partial_W w,$$
involves the differntial $\partial_V v$ and $\partial_Ww$ of the generators $v$ and $w$, derivations $\sigma_v, \sigma_w \in \text{Der}(L)$ and the above star product.
\vskip .3cm
Section 4 is dedicated to give some applications of the Quillen model of the product of $2$-cones and further examples. Section 4.1 contains an explicit model for the diagonal map $X\to X\times X$ of $2$-cones, which is a essential tool if we are interested in studying invariants such as sectional category in terms of minimal Quillen models.



Finally, in Section 4.2 we explicitly describe the minimal Quillen model of a product of two particular 3-cones and we confirm that the formula given for 2-cones is insufficient, since new terms appear that are not described in it.
\section{preliminaries}

We assume the reader is familiar with the basics of rational homotopy theory being \cite{FHT,Tanre} excellent and standard references. With the aim of fixing notation we give some definitions we will need. Throughout this paper we assume that $\mathbb{Q}$ is the base field.

A {\em graded Lie algebra} is a $\mathbb{Z}$-graded vector space $L=\oplus_{p\in \mathbb{Z}} L_p$ with a bilinear product called the {\em Lie bracket}, denoted by $[-,-]$ which satisfies graded antisymmetry, $[x, y]=-(-1)^{|x||y|}[y, x]$, and graded Jacobi identity,

\begin{equation}\label{Jacobi}
(-1)^{|x||z|}\Bigl[ x , [y, z]\Bigr] + (-1)^{|y||x|}\Bigl[ y , [z, x]\Bigr]+ (-1)^{|z||y|} \Bigl[ z, [x, y]\Bigr]=0,
\end{equation}

where $|x|$ denotes the degree of an element $x$. We say that $\sigma \colon L \to L$ is a {\em derivation of degree} $p$, and write $\sigma \in \text{Der}_p (L)$, if $\sigma (L_n)\subset L_{n+p}$  and $\sigma ([x , y]) =[\sigma x , y ]+(-1)^{|x|p}[x , \sigma y]$ for any $x, y\in L$.

A {\em differential graded Lie algebra} (dgl from now on) is a graded Lie algebra $L$ endowed with a linear derivation $\partial $ of degree $-1$ which is a differential, that is, $\partial\circ \partial =0$. A dgl $L$ is called {\em free} if it is free as a Lie algebra, $L=\mathbb{L}(V)$ for some graded vector space $V$ and it is called {\em reduced} if $L_p=0$ for $p\leq 0$. 

We denote by $L*L'$ the {\em free product} of the graded Lie algebras  $L$ and $L'$. In general, if $L$ and $L'$ are expressed in terms of generators and relations as $L=\mathbb{L}(V)/I$ and $L'=\mathbb{L}(V')/I'$, then $L*L'=\mathbb{L}(V\oplus V')/I\cup I'$.

Let $L$ be a dgl. We denote by $\text{Der}(L)=\oplus_{p\in \mathbb{Z}} \text{Der}_p (L)$ the dgl of derivations, where the Lie bracket and differential are given by
$$[\sigma , \tau ]=\sigma \circ \tau -(-1)^{|\sigma ||\tau |} \tau \circ \sigma,\ \ \ D\sigma = \partial \circ \sigma -(-1)^{|\sigma |} \sigma \circ \partial,$$
(altough we usually omit the sign $\circ$).

Let $x\in L$, we define the linear map $\text{ad}_x \colon L \to L$ of degree $|x|$ by $\text{ad}_x (y) = [x, y]$ for any $y\in L$. We can prove that Jacobi identity is equivalent to $ad_x$ being a derivation for any $x\in L$, so identity $(\ref{Jacobi})$ will be used more frequently in the form
\begin{equation*}
\Bigl[ x , [y, z]\Bigr] = \Bigl[ [x , y], z]\Bigr]+ (-1)^{|x||y|} \Bigl[ y, [x, z]\Bigr].
\end{equation*}

In \cite{Quillen}, D. Quillen constructed an equivalence
between the  homotopy category of simply connected rational complexes and the homotopy category of reduced differential graded
 Lie algebras.
 $$   \begin{array}{c}
        \text{ Simply connected}\\
        \text{ spaces}\end{array}   \xymatrix{ \ar@<0.75ex>[r]^{\lambda} &
\ar@<0.75ex>[l]^{\langle -\rangle }}\begin{array}{c}
        \text{ Reduced}\\
        \text{dgl's}\end{array}$$
We say that a reduced dgl $L$ is a {\em model} of the simply connected complex $X$ if there is a sequence of dgl quasi-isomorphisms
\begin{equation*}
L\stackrel{\simeq }{\rightarrow}\cdots \stackrel{\simeq }{\leftarrow} \lambda X.
\end{equation*}

If $L=(\mathbb{L} (V),\partial)$ is free  we say that it is a {\em Quillen model} of $X$. If $\partial(V)\subset \mathbb{L}^{\geq 2} (V)$, i.e., the differential $\partial $ has no linear term, we say that $L$ is the minimal Quillen model of $X$.

For any model one has $H(L)\cong \pi_*(\Omega X)\otimes\mathbb{Q}$ as Lie algebras and if $L=(\mathbb{L} (V),\partial)$ is a Quillen model we have $H(V,\partial_1)\cong s\widetilde H(X;\mathbb{Q})$ where $\partial_1\colon V\to V$ denotes the linear part of $\partial$ and $s$ denotes the suspension operator which is defined for any graded vector space $W$ by $(sW)_p=W_{p-1}$.

\begin{example} \label{minimalQuillenmodels}
Some examples of minimal Quillen models are given by:
\begin{itemize}
\item[(i)] Spheres $S^n$ have minimal Quillen models with a single generator $v$ in degree $n-1$ and zero differential $(\mathbb{L}(v), 0)$.

\item[(ii)] The wedge of two simply connected spaces $X$ and $Y$ with minimal Quillen models $(\mathbb{L}(V), \partial_V)$ and $(\mathbb{L}(W), \partial_W)$ respectively, is given by $(\mathbb{L}(V\oplus W), \partial )$ where $\partial (v) = \partial_V (v)$ and $\partial (w) = \partial_W (w)$ for any $v\in V$ and $w\in W$.
\item[(iii)] Any co-H-space $X$ has the rational homotopy type of a wedge of spheres $X\simeq_{\mathbb{Q}} \bigvee_{n} S^{n} $ and therefore has a minimal Quillen model of the form $(\mathbb{L}(\{ v_n\} ), 0)$ with $|v_n|=n-1$.
\item[(iv)] The minimal Quillen model of the complex projective plane $\mathbb{C}P^2$ is of the form $(\mathbb{L}(x, y), \partial )$ where $|x|=1$, $|y|=3$ with $\partial x=0$ and $\partial y =[x, x]$.
 
\end{itemize}
\end{example}

A minimal Quillen model of a Cartesian product can be given in terms of minimal Quillen models of the factors. The classic result stated in the introduction \cite[VII.1.(2)]{Tanre} admits the following modification which can be found in \cite[Proposition 1.3]{BC}.

\begin{theorem}\label{Tanre}
Let $(\mathbb{L}(V),\partial_V)$ and $(\mathbb{L}(W),\partial_W)$ be minimal Quillen models for $X$ and $Y$ respectively. Then the minimal Quillen model for $X\times Y$ has the form $\Psi \colon L= \left(\mathbb{L} (V\oplus W\oplus s(V\otimes W)), D \right)\stackrel{\simeq}{\longrightarrow} (\mathbb{L}(V),\partial_V)\times (\mathbb{L}(W),\partial_W)$, with $\Psi(v)=v$, $\Psi(w)=w$, $\Psi(s(v\otimes w))=0$, $D(v)=\partial(v)$, $D(w)=\partial(w)$ and \[D(s(v\otimes w))=[v,w]+D^+(s(v\otimes w)),\] where $D^+(s(v\otimes w))\in I_s:=\mathbb{L} (V\oplus W)*\mathbb{L}^+(s(V\otimes W))$.
\end{theorem}

From now on, we will keep $L$ to design this particular model of a Cartesian product of two spaces. 

A useful result \cite[Lemma 1.4]{BC} that we will use extensively in this work is the following:

\begin{lemma}\label{Quism}
	Let $L:=\left(\mathbb{L} (V\oplus W\oplus s(V\otimes W)), D\right)$ be as in Proposition \ref{Tanre} and $\varphi\colon L\rightarrow (\mathbb{L}(V),\partial_V)\times (\mathbb{L}(W),\partial_W)$ defined, for $v\in V$ and $w\in W$, as $\varphi(v)=v$, $\varphi(w)=w$ and $\varphi(s(v\otimes w))=0$. Then $\varphi$ is a quasi-isomorphism.
\end{lemma}

\begin{example} \label{cartesianproducts} Some examples of minimal Quillen models for Cartesian products are the following:
\begin{itemize}
\item[(i)]The minimal Quillen model of the product of spheres $S^n\times S^m$ is given by $(\mathbb{L}(v , w, s(v\otimes w)), D)$ where $|v|=n-1$, $|w|=m-1$ with $Dv=Dw=0$ and $D(s(v\otimes w))=[v,w]$.
\item[(ii)] The minimal Quillen model of the product $S^n \times \mathbb{C}P^2$ is given by $$L=(\mathbb{L}(v , x, y, s(v\otimes x), s(v\otimes y)), D),$$ where $|v|=n-1$, $|x|=1$, $|y|=3$ with $Dv=Dx=0$, $Dy =[x, x]$, $D(s(v\otimes x))=[v,x]$ and $D(s(v\otimes y))=[v , y] + 2[x, s(v\otimes x)]$.

We have describe in Example \ref{minimalQuillenmodels}, the minimal Quillen model of $S^n$ and $\mathbb{C}P^2$, so it is only necessary to check the definition of $D$ in generators $s(v\otimes x)$ and $s(v\otimes y)$.

On the one hand $D(v\otimes x)=[v , x]$ defines a differential since $Dv = Dx = 0$, but on the other hand $D(v\otimes y) = [v , y]$ would not satisfies $D^2=0$. 

Indeed, $$D[v , y]=[Dv , y]+(-1)^{n-1}[v, Dy] =-(-1)^{n}[v, [x, x]]\not= 0.$$
Nevertheless, since $[v , [x ,x]]= 2[[v, x], x]$ by Jacobi identity and $s(v\otimes x)=[v ,x]$, the definition $D(s(v\otimes y))=[v , y] + 2[x, s(v\otimes x)]$ satisfies
\begin{align*}
D^2 (s(v\otimes y))&=D[v, y]+(-1)^n 2[ x, D(s(v\otimes x))]\\
&=-(-1)^n2[[v, x], x] +(-1)^n2[[v, x], x]=0.
\end{align*}

Note that the fact that $L$ is a model for $S^n\times \mathbb{C}P^2$ is provided by Lemma \ref{Quism}.
\end{itemize}
\end{example}

Some derivations that will be essential in the description of differential $D$ in the explicit models of Cartesian products are defined as follows. As usual, the definition of a derivation of a free Lie algebra $\mathbb{L}(U)$ is given on $U$ and extended to $\mathbb{L}(U)$ by the Leibniz rule. 

\begin{definition}\label{derivations}
Let $v\in V$ and $w\in W$. Define $\sigma_ v , \sigma_w \in \text{Der} (L)$ by  $\sigma _v(w)=s(v\otimes w)$, $\sigma_w (v)=(-1)^{|v||w|}s(v\otimes w)$ and
$$\sigma_v (v')=\sigma_v(s(v'\otimes w))=\sigma_w (w')=\sigma_w (s(v\otimes w'))=0.$$
Note that $|\sigma_v | = |v|+1$.

For any $A \in \mathbb{L}(V\oplus W)$ we define $\sigma _A \in \text{Der}(L)$ by $\sigma_A (v)=(-1)^{|v||A|}\sigma_v (A)$, $\sigma_A (w)=(-1)^{|w||A|}\sigma_w(A)$ and $\sigma_A (s(v\otimes w))=0$.
\end{definition}

Example \ref{cartesianproducts}(ii) can be generalized to the following theorem due to G. Lupton and S. Smith: 

\begin{theorem}\cite[Theorem 3.3]{Lupton-Smith} \label{Lupton-Smith}
If $X$ is a co-H-space and $Y$ is any simply connected space with Quillen minimal models $(\mathbb{L}(V), 0)$ and $(\mathbb{L}(W), \partial)$ respectively, then, the minimal Quillen model of the product $X\times Y$ is given by $\Bigl(\mathbb{L}(V\oplus W \oplus s(V\otimes W)), D\Bigr)$, where $Dv=0$, $Dw=\partial w$ and
\begin{equation}\label{diferentialLuptonSmith}
D(s(v\otimes w))=[v, w]-(-1)^{|v|}\sigma_v(\partial w ).
\end{equation}
\end{theorem}

\begin{proof}
Note that $D^+(s(v\otimes w))=-(-1)^{|v|}\sigma_v(\partial w)\in I_s$ by definition of the derivation $\sigma_v$. Then, by Lemma \ref{Quism} it is only necessary to check that $D^2(s(v\otimes w))=0$. Given $v\in V$, consider the derivations $$\text{ad}_v,  [ D, \sigma_v ]\in \text{Der}(L),$$ both of degree $|v|$. 

These derivations are essentially different. Indeed, if $v'\in V$ we have
\begin{align*}
[D, \sigma_v](v')&=D \sigma_v (v') - (-1)^{|v|+1}\sigma_v D (v')=0\\
\text{ad}_v(v')&=[v, v'],
\end{align*}
where we have used  $\sigma_v (v')=0$ and $Dv'=0$.
Nevertheless, for $w\in W$ we have that
\begin{align*}
[D, \sigma_v](w)&=D \sigma_v (w) - (-1)^{|v|+1}\sigma_v D (w)\\
&=D(s(v\otimes w)) + (-1)^{|v|}\sigma_v (\partial w)\\
&=[v, w]-(-1)^{|v|}\sigma_v(\partial w)+(-1)^{|v|}\sigma_v(\partial w)\\
&=\text{ad}_v (w).
\end{align*}
Since $\text{ad}_v,  [ D, \sigma_v ]\in \text{Der}(L)$ we have that 
\begin{equation*}
[D, \sigma_v](A) =\text{ad}_v (A),\ \text{for any } A \in \mathbb{L}(W)\subset L.
\end{equation*}
Thus, as $\partial w \in \mathbb{L}(W)$ and $\partial v=0$, we can conclude that
\begin{align*}
D^2(s(v\otimes w))&=D\Bigl( \text{ad}_v(w)-(-1)^{|v|}\sigma_v (\partial w)\Bigr)\\
&=(-1)^{|v|}\text{ad}_v(\partial w)-(-1)^{|v|}D\sigma_v (\partial w)\\
&=(-1)^{|v|}[D, \sigma_v](\partial w)-(-1)^{|v|}D\sigma_v (\partial w)\\
&=(-1)^{|v|}\Bigl( D \sigma_v(\partial w) + (-1)^{|v|}\sigma_v  D(\partial w)\Bigl)-(-1)^{|v|}D\sigma_v (\partial w)\\
&=\sigma_ v (\partial \partial w)+ (-1)^{|v|}D(\sigma_v (\partial w))-(-1)^{|v|}D(\sigma_v (\partial w))=0.
\end{align*}
\end{proof}

\section{Cartesian product of $2$-cones}
The goal of this section is the explicit description of the minimal Quillen model of the Cartesian product $X\times Y$ in the case of $X$ and $Y$ being $2$-cones. In other terms, let $(\mathbb{L}(V), \partial_V)$ and $(\mathbb{L}(W), \partial_W)$ be the minimal Quillen models of $X$ and $Y$ respectively, where $V=V_0\oplus V_1$ and $W=W_0\oplus W_1$ with $\partial _V (V_0) = \partial _W (W_0)=0$ and $\partial (V_1)\subset \mathbb{L}(V_0)$ and $\partial (W_1)\subset \mathbb{L}(W_0)$. Then, we will give a formula for the differential $D$ of the minimal Quillen model of $X\times Y$ of the form $(L, D)=\Bigl(\mathbb{L}(V\oplus W \oplus s(V\otimes W)), D\Bigr)$ of Theorem \ref{Tanre}.


\subsection{Derivations and differential on $L_1$}

We start defining the differential $D$ on 
$$L_1 = \mathbb{L}(V\oplus W \oplus s(V_0 \otimes W_0) \oplus s(V_1\otimes W_0)\oplus s(V_0 \otimes W_1)),$$
in the same way as in Theorem \ref{Lupton-Smith}:

\begin{definition}\label{L1}
Let $s(v\otimes w)\in s(V_0\otimes W_0)\oplus s(V_1\otimes W_0)\oplus s(V_0\otimes W_1)$. The differential $D$ is defined by

$$
Ds(v\otimes w)= \left\{ \begin{array}{lcc}
            \ [v, w] &   \text{if}  & v\in V_0, w\in W_0, \\
\ [v, w]-(-1)^{(|v|+1)|w|}\sigma_w (\partial v) &  \text{if} & v\in V_1, w\in W_0, \\
             \ [v, w]-(-1)^{|v|}\sigma_v (\partial w) &  \text{if}  &v\in V_0, w\in W_1.
             \end{array}
   \right.
   $$
\end{definition}

The next lemma is established in the same way as Theorem \ref{Lupton-Smith}. 
\begin{lemma}
With the above definition $(L_1, D)$ is a dgl.
\end{lemma}
\begin{remark}[\bf The tree notation] \label{The tree notation}
To help the reader follow the somewhat technical proofs and tedious examples, we provide a visual idea of the formulas in terms of trees with the following convention: 

Any element of $\mathbb{L}(V\oplus W\oplus s(V\otimes W))$ can be written as a planar binary rooted tree (representing the bracketing) with leaves labelled on the elements of $V\oplus W \oplus s(V\otimes W)$.

For example, the bracket $\Bigl[ [a,b],[c,d]\Bigr]$ is represented by the tree $\stackrel{a\hskip .1cm b\hskip .1cm c\hskip .1cm d}{\hskip .1cm \RS{{LL{L}{R}}{RR{R}L}} \hskip .1cm }$.

The Jacobi identity $[a, [b, c]]=[[a, b], c]+(-1)^{|a||b|}[b, [a, c]]$ can be usefully expressed in terms of trees as follows: 

$$\stackrel{ \text{\Large $a$}\hskip .5cm \text{\Large $b$}\hskip .6cm \text{\Large $c$}}{\hskip .1cm \RS{{LLL{LLL}}{RRR{RRR}LLL}} \hskip .1cm } \hskip.2cm = \hskip .2cm\stackrel{ \hskip .1cm\text{\Large $a$}\hskip .6cm \text{\Large $b$}\hskip .5cm \text{\Large $c$}}{\hskip .1cm \RS{{LLL{LLL}RRR}{RRR{RRR}}}  }  \hskip .2cm +(-1)^{|\text{\large $a$}||\text{\large $b$}|}\stackrel{ \hskip .1cm \text{\Large $b$}\hskip .6cm \text{\Large $a$}\hskip .5cm \text{\Large $c$}}{\hskip .1cm \RS{{LLL{LLL}}{RRR{RRR}LLL}} }.$$
\vskip.3cm
If we have two elements $S, T \in \mathbb{L}(V\oplus W)$, then the derivation $\sigma _S$ applied on $T$ will be represented by 
$$
 \xymatrixcolsep{.5pc}
\xymatrixrowsep{.5pc}
\entrymodifiers={=<1pc>} \xymatrix{
 &&S\ar@{-}[d]& \\
 \sigma_S(T)&\ \ \ =&*{}\ar@{-}[d]&,\\
&&T&} $$
which can be understood recursively by the following ``grafting'' of trees:
\vskip.1cm
$$
 \xymatrixcolsep{.5pc}
\xymatrixrowsep{.5pc}
\entrymodifiers={=<1pc>} \xymatrix{
 S\ar@{-}[d]&&&&&&&&&&\hskip.3cmS\ar@{-}@<3pt>[d]&&\hskip.2cmS\ar@{-}@<2pt>[d]&&&&\hskip.2cmS\ar@{-}@<2pt>[d]\\
*{}\ar@{-}[d]&=&\ \ \sigma_S(b)&\hskip.4cm =&\hskip 2cm (-1)^{|S||b|}\sigma_b(S),&\hskip4.1cm \text{ and }&&&&&\hskip .3cm\aoverbrace[L1R]{T\ \ \ \ T'}&\hskip.7cm=&\hskip.2cm\aoverbrace[]{T}&\hskip.4cm\aoverbrace[]{T'}&\hskip.2cm\pm&\aoverbrace[]{T}&\hskip.2cm\aoverbrace[]{T'.}\\
b&&&&&&&&&&\hskip.15cm\RS{{LLL}{RRR}}&&\hskip.8cm\RS{{LLL}{RRR}}&&&\hskip.6cm\RS{{LLL}{RRR}}&}$$

Intuitively, since $\sigma_S$ is a derivation, the graft of a tree $S$ onto another tree $T$ runs through each leaf $b$ of $T$ applying $\sigma_S$, which in turn runs through each leaf $a$ of tree $S$ applying $\sigma_b$ (and recall that $\sigma_a (b)=s(a\otimes b)$ if $a\in V$ and $b\in W$, $\sigma_a (b)=(-1)^{|a||b|}s(b\otimes a)$ if $b\in V$ and $a\in W$ or $\sigma_a (b)=0$ in other case). 

\vskip.3cm
We can describe the differential of Definition \ref{L1} with tree notation as follows. When $v\in V_0$ and $w\in W_1$ we have


$$
 \xymatrixcolsep{.5pc}
\xymatrixrowsep{.5pc}
\entrymodifiers={=<1pc>} \xymatrix{
&&&{v}\ar@{-}[rd]&&\ar@{-}[ld]{w} &    &v\ar@{-}[d]&    \\
Ds(v\otimes w)&&=&&*{}\ar@{}[d]&     & + &*{}\ar@{-}[d]&.\\
&&&&& & & \partial w&
}$$



\end{remark}

\begin{lemma} \label{firstLemma}
For $S\in \mathbb{L}(V)$ and $T\in \mathbb{L}(W)$ we have the following equalities of derivations
\begin{enumerate}
\item [(i)] $[D, \sigma_S ] = \text{ad}_S -\sigma_{\partial S}$ on $\mathbb{L}(W_0)$,
\item [(ii)] $[D, \sigma_T ] = -\text{ad}_T -\sigma_{\partial T}$ on $\mathbb{L}(V_0)$.
\end{enumerate}
\end{lemma}
\begin{proof}
Let $w\in W_0$ and consider the derivations $[D, \sigma_w ], -\text{ad}_w \in \text{Der}(L_1)$. If $v\in V$ we have
\begin{align*}
[D, \sigma _w ](v)&=D \sigma_w (v)-(-1)^{|w|+1}\sigma_w D (v)\\
&=(-1)^{|v||w|}Ds(v\otimes w)+(-1)^{|w|}\sigma_w (\partial v)\\
&=(-1)^{|v||w|}[v, w]-(-1)^{|w|}\sigma_w (\partial v)+(-1)^{|w|}\sigma_w (\partial v)\\
&=(-1)^{|v||w|}[v, w]=-\text{ad}_w(v).
\end{align*}
Since $[D, \sigma_w]$ and $-\text{ad}_w$ agree in $V$, and $S \in \mathbb{L}(V)$, we have that 
\begin{equation*}
[D, \sigma_w](S) =-\text{ad}_w (S),
\end{equation*} 
which can be written as
\begin{equation} \label{previous}
D \sigma_w(S)=\Bigl( (-1)^{|w|+1}\sigma_w D -\text{ad}_w \Bigr) (S).
\end{equation}

In order to prove (i), we evaluate the derivations $[D, \sigma_S], \text{ad}_S -\sigma_{\partial S} \in \text{Der}(L_1)$, on $w\in W_0$.
\begin{align*}
[D, \sigma_S](w)&=D \sigma_S (w)-(-1)^{|S|+1}\sigma_S D (w)\\
&=(-1)^{|S||w|}D \sigma_w (S).\\
(\text{ad}_S-\sigma_{\partial S})(w)&=[S, w]-\sigma_{\partial S}(w)\\
&=-(-1)^{|S||w|}\text{ad}_w(S)-(-1)^{(|S|+1)|w|}\sigma_w (\partial S)\\
&=(-1)^{|S||w|}\Bigl( (-1)^{|w|+1}\sigma_w  D -\text{ad}_w \Bigr) (S).
\end{align*}

By equation (\ref{previous}) we have $[D, \sigma_S](w)= (\text{ad}_S-\sigma_{\partial S})(w)$ for any $w\in W_0$ and we conclude that
\begin{equation*}
[D, \sigma_S](T) =(\text{ad}_S-\sigma_{\partial S})(T), \text{ for any } T \in \mathbb{L}(W_0).
\end{equation*}
A similar computation shows that (ii) also holds.
\end{proof}

\begin{remark}\label{iii and iv} (i) and (ii) of Lemma \ref{firstLemma} are equivalent to:
\begin{enumerate}
\item[(iii)] $D(\sigma_S(T))=[S, T] -\sigma_{\partial S}(T)$, if $T\in \mathbb{L}({W_0})$,
\item[(iv)] $D(\sigma_T(S)) = (-1)^{|S||T|}[S, T] -\sigma_{\partial T}(S)$, if $S\in \mathbb{L}(V_0)$,
\end{enumerate}

respectively.
With the tree notation, case (iii) can be written as:

\begin{equation}\label{iii tree}
 \xymatrixcolsep{.5pc}
\xymatrixrowsep{.5pc}
\entrymodifiers={=<1pc>} \xymatrix{
&S\ar@{-}[dd]&&{S}\ar@{-}[rd]&&\ar@{-}[ld]{T} &    &\partial S\ar@{-}[d]&    \\
D\Bigl(&\hskip1cm\Bigr)&\hskip.5cm=&&*{}\ar@{}[d]&     &- &*{}\ar@{-}[d]&.\\
&T&&&& & & T&
}
\end{equation}
\end{remark}

\begin{remark}
The above tree notation formula can be understood in an informal visual way as follows: first, if $S\in \mathbb{L}(V_0)$ and $w\in W_0$ we can compute easily

$$
 \xymatrixcolsep{.5pc}
\xymatrixrowsep{.5pc}
\entrymodifiers={=<1pc>} \xymatrix{
&S\ar@{-}[dd]&&&&&&w\ar@{-}[dd]&&&&{w}\ar@{-}[rd]&&\ar@{-}[ld]{S}&&{S}\ar@{-}[rd]&&\ar@{-}[ld]{w}& \\
D\Bigl(&&\Bigr)=&&(-1)^{|w||S|}&&D\Bigl(&&\Bigr)=&&-(-1)^{|w||S|}&&*{}\ar@{}[d]&&=&&*{}\ar@{}[d]&&. \\
&w&&&&&&S&&&&&&&&&&&
}$$

Indeed, since $\sigma_S (w)=(-1)^{|w||S|}\sigma_w (S)$, the differential $D$ runs through the leaves $v$ of tree $S$ vanishing, until it reaches $\sigma_w (v)=(-1)^{|v||w|}s(v\otimes w)$ giving rise to $(-1)^{|v||w|}[v, w] =-\text{ad}_w (v)$. Therefore, we obtain the derivation $-\text{ad}_w$ moving through the leaves of tree $S$, that is, the formula on the right.

Next, if $S\in \mathbb{L}(V)$ and $w\in W_0$, $D$ runs through the leaves $v$ of the tree $S$ with graft $\sigma_w$, as before, but this time $Dv = \partial v$ on each ungrafted leaf of the tree until reaching $\sigma_w (v)$, giving rise to 
$$(-1)^{|w||S|}D\sigma_w (v)=(-1)^{|w|(|S|+|v|)}[v, w]-(-1)^{|w|(|S|+1)}\sigma_w (\partial v).$$

On the one hand, the term $(-1)^{|w|(|S|+|v|)}[v, w]=-(-1)^{|w||S|}\text{ad}_w (v)$ gives rise to $-(-1)^{|w||S|}\text{ad}_w (S)$ as in the previous case. On the other hand the non-vanishing terms give rise to $-(-1)^{|w|(|S|+1)}\sigma_w (\partial S)=-\sigma_{\partial S}(w)$.
$$
 \xymatrixcolsep{.5pc}
\xymatrixrowsep{.5pc}
\entrymodifiers={=<1pc>} \xymatrix{
&S\ar@{-}[dd]&&&&&&w\ar@{-}[dd]&&&&{w}\ar@{-}[rd]&&\ar@{-}[ld]{S}&&&&w\ar@{-}[dd]& \\
D\Bigl(&&\Bigr)=&&(-1)^{|w||S|}&&D\Bigl(&&\Bigr)=&&-(-1)^{|w||S|}&&*{}\ar@{}[d]&&&-(-1)^{|w|(|S|+1)}&&& \\
&w&&&&&&S&&&&&&&&&&\partial S& \\
&&&&&&&&&&&S\ar@{-}[rd]&&w\ar@{-}[ld]&&&&\partial S\ar@{-}[dd]& \\
&&&&&&&&=&&&&*{}\ar@{}[d]&&&-&&&. \\
&&&&&&&&&&&&&&&&&w&
}$$
Finally, if $S\in \mathbb{L}(V)$ and $T\in \mathbb{L}(W_0)$, in order to compute 
$$
 \xymatrixcolsep{.5pc}
\xymatrixrowsep{.5pc}
\entrymodifiers={=<1pc>} \xymatrix{
&S\ar@{-}[dd]    \\
D\Bigl(&\hskip1cm\Bigr),\\
&T
}
$$
we have the differential of tree $T$ with graft $\sigma_S$ on each leaf $w$. $D$ runs through the leaves $w$ of $T$ vanishing until it reaches $\sigma_S (w)$ giving rise to 
$$\text{ad}_S (w)-\sigma_{\partial S} (w)$$
as in the previous case. The result is precisely the derivation $\text{ad}_S$ applied to tree $T$ minus the derivation $\sigma_{\partial S}$ applied to tree $T$, that is, the formula $(\ref{iii tree})$.
\end{remark}
We now study the composition of two derivations $\sigma_A, \sigma_B\in \text{Der}(L)$ for $A, B\in \mathbb{L}(V)$.

\begin{lemma}\label{composition2derivations}
For $A, B\in \mathbb{L}(V), w\in W$ and $X\in L$ we have:
\begin{enumerate}
\item[(i)] $\sigma_A \sigma_B (w)=0,$
\item[(ii)] $\sigma_A \sigma_B(X)=(-1)^{(|A|+1)(|B|+1)}\sigma_B \sigma_A (X)$, or equivalently $[\sigma_A, \sigma_B]=0$.
\end{enumerate}
\end{lemma}

\begin{proof}
First note that $\sigma_A \sigma_B (w)=(-1)^{|w||B|}\sigma_A \sigma _w (B)$. 

If $B=v\in V$, we have $\sigma_A \sigma_w (v)=(-1)^{|v||w|}\sigma_A(s(v\otimes w))=0$.

Suppose that $\sigma_A \sigma_w (B)=0$ for any $B$ of bracket-length less or equal to $n$, that is, $B\in \mathbb{L}^{\leq n}(V)$. Let us check that $\sigma_A \sigma_w (B)=0$ holds also for $B\in \mathbb{L}^{n+1}(V)$.

If $B\in \mathbb{L}^{n+1}(V)$, $B=[T, T']$ where $T, T'\in \mathbb{L}^{\leq n}(V)$. Then
\begin{align*}
\sigma_A  \sigma_w (B)&=\sigma_A \sigma_w [T, T']\\
&=\sigma_A \Bigl( [\sigma_w (T), T']\pm [T, \sigma_w(T')]\Bigr)\\
&=[\sigma_A \sigma_w(T), T']\pm [T, \sigma_A  \sigma_w (T')]=0,
\end{align*}
where we have used that $\sigma_A(T)=0$ since $A, T\in \mathbb{L}(V)$, and the inductive hypothesis. This concludes part (i).

To prove (ii) it is only necessary to check that the derivation $[\sigma_A, \sigma_B ]\in \text{Der}(L)$ vanish on generators of $L=\mathbb{L}(V\oplus W\oplus s(V\otimes W))$.

Since $\sigma_v (v')=\sigma_v (s(v'\otimes w))=0$ for any $v , v'\in V$ and $w\in W$ we have that $[\sigma_A, \sigma_B]$ vanish on $V\oplus s(V\otimes W)$. 
If $w\in W$ by (i) we have
$[\sigma_A , \sigma_B](w)=\sigma_A  \sigma_B (w) -(-1)^{(|A|+1)(|B|+1)}\sigma_B \sigma_A (w)=0$ completing the proof.
\end{proof}

\begin{lemma}\label{secondformulalemma}
For $A, B\in \mathbb{L}(V)$ we have $\Bigl[ [ D, \sigma_A ],\sigma_B \Bigr] = (-1)^{|A|+1} \sigma_{[A, B]}$ on $\mathbb{L}(W_0)$. 

In particular, if $A, B\in \mathbb{L}(V_0), T\in \mathbb{L}(W_0)$, we have
\begin{equation}\label{secondformula}
D\Bigl( \sigma_A\sigma_B (T)\Bigr) = (-1)^{|A|+1} \sigma_{[A, B]} (T)+(-1)^{|A|+|B||T|}[\sigma_A (T), B]+[A, \sigma_B(T)].
\end{equation}
\end{lemma}
\begin{proof}
Since $\Bigl[ [ D, \sigma_A ],\sigma_B \Bigr], (-1)^{|A|+1}\sigma_{[A, B]}\in \text{Der}(L)$, it is enough to check that both derivations agree on $W_0$. Let $w\in W_0$, then on the one hand
\begin{align*}
\sigma_{[A, B]}(w)&=(-1)^{|w|(|A|+|B|)}\sigma_w ([A, B])\\
&=(-1)^{|w|(|A|+|B|)}[\sigma_w (A), B]+(-1)^{|w||B|+|A|}[A, \sigma_w (B)]\\
&=(-1)^{|w||B|}[\sigma_A (w), B]+(-1)^{|A|}[A, \sigma_B (w)].
\end{align*}
On the other hand,
\begin{align*}\Bigl[ [ D, \sigma_A ],\sigma_B \Bigr] &= [D  \sigma_A -(-1)^{|A|+1}\sigma_A  D, \sigma_B ] \\
&= D \sigma_A \sigma_B + (-1)^{|A|}\sigma_A  D \sigma_B\\
&-(-1)^{(|B|+1)|A|}\sigma_B D \sigma_A + (-1)^{|B|+|A||B|+1}\sigma_B \sigma_A D,
\end{align*}
and by (i) of Lemma \ref{composition2derivations}, (iv) of Remark \ref{iii and iv}, we have 
\begin{align*}
\Bigl[ [ D, \sigma_A ],\sigma_B \Bigr](w)&=(-1)^{|A|}\sigma_A D\sigma_B (w) -(-1)^{(|B|+1)|A|}\sigma_BD\sigma_A (w)\\
&=(-1)^{|A|}\sigma_A \Bigl( [B, w]-\sigma_{\partial B}(w)\Bigr) \\ &- (-1)^{(|B|+1)|A|}\sigma_B \Bigl( [A, w] -\sigma_{\partial A }(w)\Bigl)\\
&=(-1)^{|A|}\sigma_A [B, w]-(-1)^{|A|}\sigma_A\sigma_{\partial B}(w)\\
&-(-1)^{(|B|+1)|A|}\sigma_B[A, w]+(-1)^{(|B|+1)|A|}\sigma_B \sigma_{\partial A} (w)\\
&=-(-1)^{|A|+|B||w|}[\sigma_A (w), B] - [A, \sigma_B (w)],
\end{align*}
concluding the first part of the Lemma.

For the second part, as $A, B\in \mathbb{L}(V_0)$ we have that $\partial A =\partial B=0$. Then,
\begin{align*}
\Bigl[ [ D, \sigma_A ],\sigma_B \Bigr](T)&= D\sigma_A\sigma_B (T)\\
&+(-1)^{|A|}\sigma_A D\sigma_B (T) -(-1)^{(|B|+1)|A|}\sigma_BD\sigma_A (T)\\
&=D\sigma_A\sigma_B (T)+(-1)^{|A|}\sigma_A [B, T]- (-1)^{(|B|+1)|A|}\sigma_B [A, T]\\
&=D\sigma_A\sigma_B (T)-(-1)^{|A|+|B||T|}[\sigma_A (T), B] - [A, \sigma_B (T)].
\end{align*}
Since $\Bigl[ [ D, \sigma_A ],\sigma_B \Bigr](T)=(-1)^{|A|+1}\sigma_{[A, B]}(T)$ for $T\in \mathbb{L}(W_0)$, we obtain formula (\ref{secondformula}).
\end{proof}

\begin{remark}
Although we have that $\sigma_A\sigma_B (w)=0$ for any $w\in W$ by Lemma \ref{composition2derivations} (i)  this does not imply that $\sigma_A\sigma_B (T)=0$ for $T\in \mathbb{L}(W)$ as $\sigma_A\sigma_B$ is not a derivation. The most we can say is $\sigma_A \sigma_B(T)=(-1)^{(|A|+1)(|B|+1)}\sigma_B \sigma_A (T)$ as in Lemma \ref{composition2derivations} (ii). The same applies for the terms $\sigma_A \sigma_{\partial B} (T)$ and $\sigma_B \sigma_{\partial A} (T)$ that does not appear in the computations of the second part of the proof of formula ($\ref{secondformula}$) because $\partial A = \partial B = 0$ when $A, B\in \mathbb{L}(V_0)$. 
\end{remark}

\begin{remark}
Formula $(\ref{secondformula})$ can be written in tree notation as:

 $$
 \xymatrixcolsep{.5pc}
\xymatrixrowsep{.5pc}
\entrymodifiers={=<1pc>} \xymatrix{
&\hskip .3cm\text{ \Large $A$ }\ar@{-}[ddr(0.8)]&&\text{ \Large $B$ }\ar@{-}[ddl(0.8)]&&\hskip.4cm\text{ \Large $A$ }\ar@{-}[dr(1)]&\hskip 1cm\text{ \Large $B$ }&{}\ar@{-}[dl(1)]&&\hskip .4cm\text{ \Large $A$ }&&&&&&\text{ \Large $B$ }\\
\text{\huge $D$}\Bigl(&&&&\hskip .2cm\Bigr)=\ \pm&&{}&&\pm&&&&+&&&\\
&&\aoverbrace[]{\text{ \huge $T$ }}&&&&\aoverbrace[]{\text{ \huge $T$ }}\ar@{-}[u(1)]&&&\hskip .4cm\text{ \Large $T$ }\ar@{-}[dr(1)]\ar@{-}@<-1ex>[uu(.8)]&&\text{ \Large $B$ }\ar@{-}[dl(1)]&&\hskip.1cm\text{ \Large $A$ }\ar@{-}[dr(1)]&&\text{ \Large $T.$}\ar@{-}[dl(1)]\ar@{-}[uu(.8)]\\
&&&&&&&&&&&&&&&}$$ 
The symbol $\begin{matrix}\stackrel{\text{\footnotesize{$A$}}\ \ \ \text{\footnotesize{$B$}} }{\RS{L}\ \ \RS{R}}  \\ \text{\Large{$T$}} \end{matrix}$ represents $\sigma_A\sigma_B (T)$. The fact that the derivations $\sigma_A$ and $\sigma_B$ are grafted on the leaves of $T$ at the same ``level'' is due to the fact that by Lemma \ref{composition2derivations} (ii) we have $\sigma_A\sigma_B (T)=\pm \sigma_B\sigma_A (T)$.

We can derive the above formula directly in tree notation using antisymmetry, Jacobi identity and formula $(\ref{iii tree})$ for $S\in \mathbb{L}(V_0)$,  i.e.
\begin{equation*}
 \xymatrixcolsep{.5pc}
\xymatrixrowsep{.5pc}
\entrymodifiers={=<1pc>} \xymatrix{
&S\ar@{-}[dd]&&{S}\ar@{-}[rd]&&\ar@{-}[ld]{T}  \\
D\Bigl(&\hskip1cm\Bigr)&\hskip.5cm=&&*{}\ar@{}[d]&.   \\
&T&&&&
}
\end{equation*}
We will proceed by an inductive argument on the number of leaves of the tree $T$. 

If $T = \xymatrixcolsep{.5pc}
\xymatrixrowsep{-.2pc}
\entrymodifiers={=<1pc>} \xymatrix{\hskip .7cm x\hskip .5cm y&\\ \hskip .6cm\RS{{LLL}{RRR}}&} $ with $x, y \in W_0$,

$$\xymatrixcolsep{.03pc}
\xymatrixrowsep{.03pc}
\entrymodifiers={=<1pc>} \xymatrix{&A\ar@{-}[rdd(1)]&&&&B\ar@{-}[ldd(1)]&&&&A\ar@{-}[dd(1)]&&B\ar@{-}[dd(1)]&&&&B\ar@{-}[dd(1)]&&A\ar@{-}[dd(1)]&\\
&&&&&&&&&&&&&&&&&&\\
D\Bigl(&&&\stackrel{}{\overbrace{\hskip 1cm}}&&&\Bigr)&=&D\Bigl(&*{}\ar@{-}[d]&&*{}\ar@{-}[d]&\hskip -.2cm\Bigr)&\hskip -.2cm\pm&D\Bigl(&*{}\ar@{-}[d]&&*{}\ar@{-}[d]&\hskip -.2cm\Bigr)\\
&&x\ar@{-}[rd(1)]&&y\ar@{-}[ld(1)]&&&&&x\ar@{-}[rd(1)]&&y\ar@{-}[ld(1)]&&&&x\ar@{-}[rd(1)]&&y\ar@{-}[ld(1)]&\\
&&&&&&&&&&&&&&&&&&
}$$
\vskip .2cm
$$\xymatrixcolsep{.03pc}
\xymatrixrowsep{.03pc}
\entrymodifiers={=<1pc>} \xymatrix{ &&&&&B\ar@{-}[dd]&&A\ar@{-}[dd]&&&&&&&&&&A\ar@{-}[dd]&&B\ar@{-}[dd]&&&& \\
(\ref{iii tree})&&&&&&&&&&&&&&&&&&&&&&& \\
=&A\ar@{-}[ddrr(1)]&&x\ar@{-}[dl(1)]&&y\ar@{-}[ddll(1)]&\pm&x\ar@{-}[ddrr(1)]&&B\ar@{-}[dr(1)]&&y\ar@{-}[ddll(1)]&\pm&B\ar@{-}[ddrr(1)]&&x\ar@{-}[dl(1)]&&y\ar@{-}[ddll(1)]&\pm&x\ar@{-}[ddrr(1)]&&A\ar@{-}[dr(1)]&&y\ar@{-}[ddll(1)] \\
&&&&&&&&&&&&&&&&&&&&&&& \\
&&&&&&&&&&&&&&&&&&&&&&& 
}$$
\vskip .2cm
$$\xymatrixcolsep{.03pc}
\xymatrixrowsep{.03pc}
\entrymodifiers={=<1pc>} \xymatrix{ &&&&&B\ar@{-}[dd]&&&&B\ar@{-}[dd]&&&&A\ar@{-}[dd]&&&&&&&&A\ar@{-}[dd]&& \\
(\text{Jacobi})&&&&&&&&&&&&&&&&&&&&&&& \\
=&A\ar@{-}[ddrr(1)]&&x\ar@{-}[dr(1)]&&y\ar@{-}[ddll(1)]&\pm&A\ar@{-}[ddrr(1)]&&y\ar@{-}[dl(1)]&&x\ar@{-}[ddll(1)]&\pm&x\ar@{-}[ddrr(1)]&&B\ar@{-}[dl(1)]&&y\ar@{-}[ddll(1)]&\pm&B\ar@{-}[ddrr(1)]&&x\ar@{-}[dr(1)]&&y\ar@{-}[ddll(1)] \\
&&&&&&&&&&&&&&&&&&&&&&& \\
&\text{\footnotesize $(a)$}&&&&&&\text{\footnotesize $(b)$}&&&&&&\text{\footnotesize $(c)$}&&&&&&\text{\footnotesize $(d)$}&&&& 
}$$
\vskip .2cm
$$\xymatrixcolsep{.03pc}
\xymatrixrowsep{.03pc}
\entrymodifiers={=<1pc>} \xymatrix{ &&&&&A\ar@{-}[dd]&&&&A\ar@{-}[dd]&&&&B\ar@{-}[dd]&&&&&&&&B\ar@{-}[dd]&& \\
&&&&&&&&&&&&&&&&&&&&&&& \\
\pm&B\ar@{-}[ddrr(1)]&&x\ar@{-}[dr(1)]&&y\ar@{-}[ddll(1)]&\pm&B\ar@{-}[ddrr(1)]&&y\ar@{-}[dl(1)]&&x\ar@{-}[ddll(1)]&\pm&x\ar@{-}[ddrr(1)]&&A\ar@{-}[dl(1)]&&y\ar@{-}[ddll(1)]&\pm&A\ar@{-}[ddrr(1)]&&x\ar@{-}[dr(1)]&&y\ar@{-}[ddll(1)] \\
&&&&&&&&&&&&&&&&&&&&&&& \\
&\text{\footnotesize $(d)$}&&&&&&\text{\footnotesize $(b)$}&&&&&&\text{\footnotesize $(c)$}&&&&&&\text{\footnotesize $(a)$}&&&& 
}$$
\vskip .2cm
$$\xymatrixcolsep{.03pc}
\xymatrixrowsep{.03pc}
\entrymodifiers={=<1pc>} \xymatrix{ &&A\ar@{-}[dr(1)]&&B\ar@{-}[dl(1)]&&&A\ar@{-}[dr(1)]&&B\ar@{-}[dl(1)]&&&&&&&A\ar@{-}[d(1)]&&&&&&B\ar@{-}[d(1)]& \\
&&&*{}\ar@{-}[d]&&&&&*{}\ar@{-}[d]&&&&&&&&\stackrel{}{\overbrace{\hskip 1cm}}&&&&&&\stackrel{}{\overbrace{\hskip 1cm}}& \\
=&\pm&&y\ar@{-}[dr(1)]&&x\ar@{-}[dl(1)]&\pm&&x\ar@{-}[dr(1)]&&y\ar@{-}[dl(1)]&&\pm&B\ar@{-}[ddrr(1)]&&x\ar@{-}[dr(1)]&&y\ar@{-}[ddll(1)]&\pm&A\ar@{-}[ddrr(1)]&&x\ar@{-}[dr(1)]&&y\ar@{-}[ddll(1)] \\
&&&&&&&&&&&&&&&&&&&&&&& \\
&\text{\footnotesize $(b)$}&&&&&&\text{\footnotesize $(c)$}&&&&&&\text{\footnotesize $(d)$}&&&&&&\text{\footnotesize $(a)$}&&&& 
}$$
\vskip.2cm
$$\xymatrixcolsep{.03pc}
\xymatrixrowsep{.03pc}
\entrymodifiers={=<1pc>} \xymatrix{ &&&&&A\ar@{-}[dr(1)]&&B\ar@{-}[dl(1)]&&&&&A\ar@{-}[d(1)]&&&&&&&&&&B\ar@{-}[d(1)]& \\
&&&&&&*{}\ar@{-}[d]&&&&&&\stackrel{}{\overbrace{\hskip 1cm}}&&&&&&&&&&\stackrel{}{\overbrace{\hskip 1cm}}& \\
=&&&\pm&&&\stackrel{}{\overbrace{\hskip 1cm}}&&&\pm&&x\ar@{-}[ddrr(1)]&&y\ar@{-}[dl(1)]&&B\ar@{-}[ddll(1)]&&+&&A\ar@{-}[ddrr(1)]&&x\ar@{-}[dr(1)]&&y\ar@{-}[ddll(1)] \\
&&&&&x\ar@{-}[dr(1)]&&y\ar@{-}[dl(1)]&&&&&&&&&&&&&&&& \\
&&&&&&&&&&&&&&&&&&&&&&& 
}$$

Next, if $T = \xymatrixcolsep{.5pc}
\xymatrixrowsep{-.2pc}
\entrymodifiers={=<1pc>} \xymatrix{\hskip .7cm X\hskip .5cm y&\\ \hskip .6cm\RS{{LLL}{RRR}}&} $ with $X\in \mathbb{L}^{\geq 2}(W_0)$ and $y \in W_0$, by the one hand
$$\xymatrixcolsep{.03pc}
\xymatrixrowsep{.03pc}
\entrymodifiers={=<1pc>} \xymatrix{&A\ar@{-}[rdd(1)]&&&&B\ar@{-}[ldd(1)]&&&&A\ar@{-}[rddd(0.8)]&&B\ar@{-}[lddd(0.8)]&&&&&&A\ar@{-}[dd(1)]&&B\ar@{-}[dd(1)]&&&&B\ar@{-}[dd(1)]&&A\ar@{-}[dd(1)]&\\
&&&&&&&&&&&&&&&&&&&&&&&&&&\\
D\Bigl(&&&\stackrel{}{\overbrace{\hskip 1cm}}&&&\Bigr)&=&D\Bigl(&&&&&&\hskip -.2cm\Bigr)&\hskip -.2cm\pm&D\Bigl(&*{}\ar@{-}[d]&&*{}\ar@{-}[d]&\hskip -.2cm\Bigr)&\hskip -.2cm\pm&D\Bigl(&*{}\ar@{-}[d]&&*{}\ar@{-}[d]&\hskip -.2cm\Bigr)\\
&&X\ar@{-}[rd(1)]&&y\ar@{-}[ld(1)]&&&&&&X\ar@{-}[rd(1)]&&y\ar@{-}[ld(1)]&&&&&X\ar@{-}[rd(1)]&&y\ar@{-}[ld(1)]&&&&X\ar@{-}[rd(1)]&&y\ar@{-}[ld(1)]&\\
&&&&&&&&&&&&&&&&&&&&&&&&&&
}$$
$$\xymatrixcolsep{.03pc}
\xymatrixrowsep{.03pc}
\entrymodifiers={=<1pc>} \xymatrix{ \text{(IH)}&&&&A\ar@{-}[dr(1)]&&B\ar@{-}[dl(1)]&&&&&A\ar@{-}[dd(0.7)]&&&&&&&&&&B\ar@{-}[dd(0.7)]&& \\
(\ref{iii tree})&&&&&*{}\ar@{-}[dd]&&&&&&&&&&&&&&&&&& \\
=&&&\pm&&&&&&\pm&&X\ar@{-}[ddrr(1)]&&B\ar@{-}[dl(1)]&&y\ar@{-}[ddll(1)]&&+&&A\ar@{-}[ddrr(1)]&&X\ar@{-}[dl(1)]&&y\ar@{-}[ddll(1)] \\
&&&&&X\ar@{-}[dr(1)]&&y\ar@{-}[dl(1)]&&&&&&&&&&&&&&&& \\
&&&&\text{\footnotesize $(a)$}&&&&&&&\text{\footnotesize $(b)$}&&&&&&&&\text{\footnotesize $(c)$}&&&& 
}$$
\vskip .05cm
$$\xymatrixcolsep{.03pc}
\xymatrixrowsep{.03pc}
\entrymodifiers={=<1pc>} \xymatrix{ &&&&&B\ar@{-}[dd]&&A\ar@{-}[dd]&&&&&&&&&&A\ar@{-}[dd]&&B\ar@{-}[dd]&&&& \\
&&&&&&&&&&&&&&&&&&&&&&& \\
\pm&A\ar@{-}[ddrr(1)]&&X\ar@{-}[dl(1)]&&y\ar@{-}[ddll(1)]&\pm&X\ar@{-}[ddrr(1)]&&B\ar@{-}[dr(1)]&&y\ar@{-}[ddll(1)]&\pm&B\ar@{-}[ddrr(1)]&&X\ar@{-}[dl(1)]&&y\ar@{-}[ddll(1)]&\pm&X\ar@{-}[ddrr(1)]&&A\ar@{-}[dr(1)]&&y,\ar@{-}[ddll(1)]\\
&&&&&&&&&&&&&&&&&&&&&&& \\
&\text{\footnotesize $(d)$}&&&&&&\text{\footnotesize $(b)$}&&&&&&\text{\footnotesize $(e)$}&&&&&&\text{\footnotesize $(c)$}&&&& 
}$$
\vskip .2cm
and using Jacobi identity (and antisymmetry) we have:
$$\xymatrixcolsep{.03pc}
\xymatrixrowsep{.03pc}
\entrymodifiers={=<1pc>} \xymatrix{\text{\hskip -.6cm\large (b) } &A\ar@{-}[dd]&&&&&&A\ar@{-}[dd]&&&&&&&A\ar@{-}[dd]&&&&&&&&&& \\
&&&&&&&&&&&&&&&&&&&&&&&& \\
\pm&X\ar@{-}[ddrr(1)]&&B\ar@{-}[dl(1)]&&y\ar@{-}[ddll(1)]&\pm&X\ar@{-}[ddrr(1)]&&B\ar@{-}[dr(1)]&&y\ar@{-}[ddll(1)]&=&\pm&X\ar@{-}[ddrr(1)]&&y\ar@{-}[dl(1)]&&B\ar@{-}[ddll(1)]&&&&&&\\
&&&&&&&&&&&&&&&&&&&&&&&& \\
&&&&&&&&&&&&&&&&&&&&&&&& 
}$$
$$\xymatrixcolsep{.03pc}
\xymatrixrowsep{.03pc}
\entrymodifiers={=<1pc>} \xymatrix{ \text{\hskip -.6cm\large (c) } &&&B\ar@{-}[dd]&&&&B\ar@{-}[dd]&&&&&&&B\ar@{-}[dd]&&&&&&&&&& \\
&&&&&&&&&&&&&&&&&&&&&&&& \\
\pm&A\ar@{-}[ddrr(1)]&&X\ar@{-}[dl(1)]&&y\ar@{-}[ddll(1)]&\pm&X\ar@{-}[ddrr(1)]&&A\ar@{-}[dr(1)]&&y\ar@{-}[ddll(1)]&=&\pm&X\ar@{-}[ddrr(1)]&&y\ar@{-}[dl(1)]&&A\ar@{-}[ddll(1)]&&&&&& \\
&&&&&&&&&&&&&&&&&&&&&&&& \\
&&&&&&&&&&&&&&&&&&&&&&&& 
}$$
$$\xymatrixcolsep{.03pc}
\xymatrixrowsep{.03pc}
\entrymodifiers={=<1pc>} \xymatrix{ \text{\hskip -.6cm\large (d) } &&&&&B\ar@{-}[dd]&&&&&&&B\ar@{-}[dd]&&&&B\ar@{-}[dd]&&&&&&&& \\
&&&&&&&&&&&&&&&&&&&&&&&& \\
\pm&A\ar@{-}[ddrr(1)]&&X\ar@{-}[dl(1)]&&y\ar@{-}[ddll(1)]&=&\pm& A\ar@{-}[ddrr(1)]&&X\ar@{-}[dr(1)]&&y\ar@{-}[ddll(1)]&\pm&A\ar@{-}[ddrr(1)]&&y\ar@{-}[dl(1)]&&X\ar@{-}[ddll(1)]&&&&&& \\
&&&&&&&&&&&&&&&&&&&&&&&& \\
&&&&&&&&\text{\footnotesize $(d_1)$}&&&&&&\text{\footnotesize $(d_2)$}&&&&&&&&&& 
}$$
$$\xymatrixcolsep{.03pc}
\xymatrixrowsep{.03pc}
\entrymodifiers={=<1pc>} \xymatrix{ \text{\hskip -.6cm\large (e) } &&&&&A\ar@{-}[dd]&&&&&&&A\ar@{-}[dd]&&&&A\ar@{-}[dd]&&&&&&&& \\
&&&&&&&&&&&&&&&&&&&&&&&& \\
\pm&B\ar@{-}[ddrr(1)]&&X\ar@{-}[dl(1)]&&y\ar@{-}[ddll(1)]&=&\pm& B\ar@{-}[ddrr(1)]&&X\ar@{-}[dr(1)]&&y\ar@{-}[ddll(1)]&\pm&B\ar@{-}[ddrr(1)]&&y\ar@{-}[dl(1)]&&X.\ar@{-}[ddll(1)]&&&&&& \\
&&&&&&&&&&&&&&&&&&&&&&&& \\
&&&&&&&&\text{\footnotesize $(e_1)$}&&&&&&\text{\footnotesize $(e_2)$}&&&&&&&&&& 
}$$
By the other hand
$$\xymatrixcolsep{.03pc}
\xymatrixrowsep{.03pc}
\entrymodifiers={=<1pc>} \xymatrix{ A\ar@{-}[dr(1)]&&B\ar@{-}[dl(1)]&&&A\ar@{-}[dr(1)]&&B\ar@{-}[dl(1)]&&&&&&A\ar@{-}[dr(1)]&&B\ar@{-}[dl(1)]&&&&&&& \\
&*{}\ar@{-}[d]&&&&&*{}\ar@{-}[d]&&&&&&&&*{}\ar@{-}[d]&&&&&&&& \\
\hskip -.6cm\pm&\stackrel{}{\overbrace{\hskip 1cm}}&&\hskip .3cm=&&\hskip -.2cm\pm&*{}\ar@{-}[d(.4)]&&&&\hskip -.2cm\pm&&&&*{}\ar@{-}[d(.4)]&&&&&&&& \\
X\ar@{-}[dr(1)]&&y\ar@{-}[dl(1)]&&&&X\ar@{-}[dr(1)]&&y\ar@{-}[dl(1)]&&&&X\ar@{-}[dr(1)]&&y\ar@{-}[dl(1)]&&&& \\
&&&&&&&&&&&&&&&&&&&&&& 
}$$
$$\xymatrixcolsep{.03pc}
\xymatrixrowsep{.03pc}
\entrymodifiers={=<1pc>} \xymatrix{ &&&&&A\ar@{-}[dr(1)]&&B\ar@{-}[dl(1)]&&&&&&&y\ar@{-}[d(1)]&&&&&&&& \\
&&&&&&*{}\ar@{-}[d]&&&&&&&&\stackrel{}{\overbrace{\hskip 1cm}}&&&&&&&& \\
&&&\hskip .3cm=&&\hskip -.2cm\pm&*{}\ar@{-}[d(.4)]&&&&\hskip -.4cm\pm&X\ar@{-}[ddrr(1)]&&A\ar@{-}[dr(1)]&&B\ar@{-}[ddll(1)]&&&&&&& \\
&&&&&&X\ar@{-}[dr(1)]&&y\ar@{-}[dl(1)]&&&&&&&&&&&&&& \\
&&&&&&&&&&&&&&&&&&&&&& 
}$$
$$\xymatrixcolsep{.03pc}
\xymatrixrowsep{.03pc}
\entrymodifiers={=<1pc>} \xymatrix{ &&&&&A\ar@{-}[dr(1)]&&B\ar@{-}[dl(1)]&&&&&&y\ar@{-}[dd(.7)]&&&&&&&&&y\ar@{-}[dd(.7)] \\
&&&&&&*{}\ar@{-}[d]&&&&&&&&&&&&&&&& \\
&&&\hskip .3cm=&&\hskip -.2cm\pm&*{}\ar@{-}[d(.4)]&&&&\hskip -.4cm\pm&X\ar@{-}[ddrr(1)]&&A\ar@{-}[dr(1)]&&B\ar@{-}[ddll(1)]&\hskip .5cm\pm&&X\ar@{-}[ddrr(1)]&&A\ar@{-}[dr(1)]&&B\ar@{-}[ddll(1)] \\
&&&&&&X\ar@{-}[dr(1)]&&y\ar@{-}[dl(1)]&&&&&&&&&&&&&& \\
&&&&&&&&&&&&&&&&&&&&&&
}$$
$$\xymatrixcolsep{.03pc}
\xymatrixrowsep{.03pc}
\entrymodifiers={=<1pc>} \xymatrix{ &&&&&A\ar@{-}[dr(1)]&&B\ar@{-}[dl(1)]&&&&&&A\ar@{-}[dd(.7)]&&&&&&&&&B\ar@{-}[dd(.7)] \\
&&&&&&*{}\ar@{-}[d]&&&&&&&&&&&&&&&& \\
&&&\hskip .3cm=&&\hskip -.2cm\pm&*{}\ar@{-}[d(.4)]&&&&\hskip -.4cm\pm&X\ar@{-}[ddrr(1)]&&y\ar@{-}[dr(1)]&&B\ar@{-}[ddll(1)]&\hskip .5cm\pm&&X\ar@{-}[ddrr(1)]&&A\ar@{-}[dr(1)]&&y\ar@{-}[ddll(1)] \\
&&&&&&X\ar@{-}[dr(1)]&&y\ar@{-}[dl(1)]&&&&&&&&&&&&&& \\
&&&&&\text{\footnotesize $(a)$}&&&&&&\text{\footnotesize $(e_2)$}&&&&&&&\text{\footnotesize $(d_2)$}&&&&
}$$
$$\xymatrixcolsep{.03pc}
\xymatrixrowsep{.03pc}
\entrymodifiers={=<1pc>} \xymatrix{ &A\ar@{-}[dd(.4)]&&&&&&&A\ar@{-}[dd(.7)]&&&&&&&&&A\ar@{-}[dd(.7)]&& &&&\\
&\stackrel{}{\overbrace{\hskip 1cm}}&&&&&&&&&&&&&&&&&&&&&&&&&&& \\
X\ar@{-}[ddrr(1)]&&y\ar@{-}[dl(1)]&&B\ar@{-}[ddll(1)]&&=&\pm&X\ar@{-}[ddrr(1)]&&y\ar@{-}[dl(1)]&&B\ar@{-}[ddll(1)]&\hskip .5cm\pm&&X\ar@{-}[ddrr(1)]&&y\ar@{-}[dl(1)]&&B\ar@{-}[ddll(1)] &&&&&&&&\\
&&&&&&&&&&&&&&&&&&&&&&&&&&&& \\
&&&&&&&&\text{\footnotesize $(b)$}&&&&&&&\text{\footnotesize $(e_1)$}&&&&&&&&&&&&&
}$$
$$\xymatrixcolsep{.03pc}
\xymatrixrowsep{.03pc}
\entrymodifiers={=<1pc>} \xymatrix{ &&&B\ar@{-}[dd(.4)]&&&&&&&B\ar@{-}[dd(.7)]&&&&&&&&&B\ar@{-}[dd(.7)]&&&\\
&&&\stackrel{}{\overbrace{\hskip 1cm}}&&&&&&&&&&&&&&&&&&&&&&&&& \\
A\ar@{-}[ddrr(1)]&&X\ar@{-}[dr(1)]&&y\ar@{-}[ddll(1)]&&=&\pm&A\ar@{-}[ddrr(1)]&&X\ar@{-}[dr(1)]&&y\ar@{-}[ddll(1)]&\hskip .5cm\pm&&A\ar@{-}[ddrr(1)]&&X\ar@{-}[dr(1)]&&y.\ar@{-}[ddll(1)] &&&&&&&&\\
&&&&&&&&&&&&&&&&&&&&&&&&&&&& \\
&&&&&&&&\text{\footnotesize $(c)$}&&&&&&&\text{\footnotesize $(d_1)$}&&&&&&&&&&&&&
}$$
Therefore, we conclude that
$$\xymatrixcolsep{.03pc}
\xymatrixrowsep{.03pc}
\entrymodifiers={=<1pc>} \xymatrix{
&&&&&&&&&&&&&&&A\ar@{-}[rd(1)]&&B\ar@{-}[ld(1)]&&&A\ar@{-}[dd(0.8)]&&&&&&&&B\ar@{-}[dd(0.8)]&\\
&A\ar@{-}[rdd(1)]&&&&B\ar@{-}[ldd(1)]&&&&&&&&&&&*{}\ar@{-}[d(.6)]&&&&&&&&&&&&&\\
&&&&&&&&&&&&
&&&&\stackrel{}{\overbrace{\hskip 1cm}}
&&&&\stackrel{}{\overbrace{\hskip 1cm}}&&&&&&&&\stackrel{}{\overbrace{\hskip 1cm}}&\\
D\Bigl(&&&\stackrel{}{\overbrace{\hskip 1cm}}&&&\Bigr)&\hskip -.2cm=&&&\text{\footnotesize (a)+(b)+(c)+(d)}&&&\hskip .2cm=&\pm&X\ar@{-}[dr(1)]&&y\ar@{-}[dl(1)]&\pm&X\ar@{-}[ddrr(1)]&&y\ar@{-}[dl(1)]&&B\ar@{-}[ddll(1)]&\pm&A\ar@{-}[ddrr(1)]&&X\ar@{-}[dr(1)]&&y.\ar@{-}[ddll(1)]\\
&&X\ar@{-}[rd(1)]&&y\ar@{-}[ld(1)]&&&&&&&&&&&&&&&&&&&&&&&&&\\
&&&&&&&&&&&&&&&&&&&&&&&&&&&&&
}$$

Finally, if $T = \xymatrixcolsep{.5pc}
\xymatrixrowsep{-.2pc}
\entrymodifiers={=<1pc>} \xymatrix{\hskip .7cm X\hskip .5cm Y&\\ \hskip .6cm\RS{{LLL}{RRR}}&} $ with $X, Y \in \mathbb{L}^{\geq 2}(W_0)$, a similar computation  gives:

$$\xymatrixcolsep{.03pc}
\xymatrixrowsep{.03pc}
\entrymodifiers={=<1pc>} \xymatrix{&A\ar@{-}[rdd(1)]&&&&B\ar@{-}[ddl(1)]&&&&&A\ar@{-}[dddr(0.8)]&&B\ar@{-}[dddl(0.8)]&&&&&&A\ar@{-}[dddr(0.8)]&&B\ar@{-}[dddl(0.8)]&\\
&&&&&&&&&&&&&&&&&&&&&\\
D\Bigl(&&&\stackrel{}{\overbrace{\hskip 1cm}}&&&\Bigr)&=&\pm&D\Bigl(&&&&&\hskip -.2cm\Bigr)&\hskip -.2cm\pm&D\Bigl(&&&&&\hskip -.4cm\Bigr)\\
&&X\ar@{-}[rd(1)]&&Y\ar@{-}[ld(1)]&&&&&&&X\ar@{-}[rd(1)]&&Y\ar@{-}[ld(1)]&&&&X\ar@{-}[rd(1)]&&Y\ar@{-}[ld(1)]&&\\
&&&&&&&&&&&&&&&&&&&&&
}$$
$$\xymatrixcolsep{.03pc}
\xymatrixrowsep{.03pc}
\entrymodifiers={=<1pc>} \xymatrix{&&&&&&&&&&A\ar@{-}[dd(1)]&&B\ar@{-}[dd(1)]&&&&B\ar@{-}[dd(1)]&&A\ar@{-}[dd(1)]&&\\
&&&&&&&&&&&&&&&&&&&&\\
&&&&&&&&\pm&D\Bigl(&*{}\ar@{-}[d]&&*{}\ar@{-}[d]&\hskip -.2cm\Bigr)&\hskip -.2cm\pm&D\Bigl(&*{}\ar@{-}[d]&&*{}\ar@{-}[d]&\hskip -.2cm\Bigr)&\\
&&&&&&&&&&X\ar@{-}[rd(1)]&&Y\ar@{-}[ld(1)]&&&&X\ar@{-}[rd(1)]&&Y\ar@{-}[ld(1)]&&\\
&&&&&&&&&&&&&&&&&&&&
}$$
$$\xymatrixcolsep{.03pc}
\xymatrixrowsep{.03pc}
\entrymodifiers={=<1pc>} \xymatrix{ \text{(IH)}&&&&A\ar@{-}[dr(1)]&&B\ar@{-}[dl(1)]&&&&&A\ar@{-}[dd(0.7)]&&&&&&&&&&B\ar@{-}[dd(0.7)]&& \\
(\ref{iii tree})&&&&&*{}\ar@{-}[dd]&&&&&&&&&&&&&&&&&& \\
=&&&\pm&&&&&&\pm&&X\ar@{-}[ddrr(1)]&&B\ar@{-}[dl(1)]&&Y\ar@{-}[ddll(1)]&&+&&A\ar@{-}[ddrr(1)]&&X\ar@{-}[dl(1)]&&Y\ar@{-}[ddll(1)] \\
&&&&&X\ar@{-}[dr(1)]&&Y\ar@{-}[dl(1)]&&&&&&&&&&&&&&&& \\
&&&&\text{\footnotesize $(a)$}&&&&&&&\text{\footnotesize $(b)$}&&&&&&&&\text{\footnotesize $(c)$}&&&& 
}$$
$$\xymatrixcolsep{.03pc}
\xymatrixrowsep{.03pc}
\entrymodifiers={=<1pc>} \xymatrix{ &&&&&&A\ar@{-}[dr(1)]&&B\ar@{-}[dl(1)]&&&&&A\ar@{-}[dd(0.7)]&&&&&&&&&&B\ar@{-}[dd(0.7)] \\
&&&&&&&*{}\ar@{-}[dd]&&&&&&&&&&&&&&&& \\
&&&\pm&&&&&&\pm&&X\ar@{-}[ddrr(1)]&&Y\ar@{-}[dr(1)]&&B\ar@{-}[ddll(1)]&&+&&X\ar@{-}[ddrr(1)]&&A\ar@{-}[dr(1)]&&Y\ar@{-}[ddll(1)] \\
&&&&&X\ar@{-}[dr(1)]&&Y\ar@{-}[dl(1)]&&&&&&&&&&&&&&&& \\
&&&&\text{\footnotesize $(a)$}&&&&&&&\text{\footnotesize $(d)$}&&&&&&&&\text{\footnotesize $(e)$}&&&& 
}$$
$$\xymatrixcolsep{.03pc}
\xymatrixrowsep{.03pc}
\entrymodifiers={=<1pc>} \xymatrix{ &&&&&B\ar@{-}[dd]&&A\ar@{-}[dd]&&&&&&&&&&A\ar@{-}[dd]&&B\ar@{-}[dd]&&&& \\
&&&&&&&&&&&&&&&&&&&&&&& \\
\pm&A\ar@{-}[ddrr(1)]&&X\ar@{-}[dl(1)]&&Y\ar@{-}[ddll(1)]&\pm&X\ar@{-}[ddrr(1)]&&B\ar@{-}[dr(1)]&&Y\ar@{-}[ddll(1)]&\pm&B\ar@{-}[ddrr(1)]&&X\ar@{-}[dl(1)]&&Y\ar@{-}[ddll(1)]&\pm&X\ar@{-}[ddrr(1)]&&A\ar@{-}[dr(1)]&&Y\ar@{-}[ddll(1)] \\
&&&&&&&&&&&&&&&&&&&&&&& \\
&\text{\footnotesize $(e)$}&&&&&&\text{\footnotesize $(b)$}&&&&&&\text{\footnotesize $(d)$}&&&&&&\text{\footnotesize $(c)$}&&&& 
}$$
$$\xymatrixcolsep{.03pc}
\xymatrixrowsep{.03pc}
\entrymodifiers={=<1pc>} \xymatrix{ &&&&&A\ar@{-}[dr(1)]&&B\ar@{-}[dl(1)]&&&&&A\ar@{-}[d(1)]&&&&&&&&&&B\ar@{-}[d(1)]& \\
&&&&&&*{}\ar@{-}[d]&&&&&&\stackrel{}{\overbrace{\hskip 1cm}}&&&&&&&&&&\stackrel{}{\overbrace{\hskip 1cm}}& \\
=&&&\pm&&&\stackrel{}{\overbrace{\hskip 1cm}}&&&\pm&&X\ar@{-}[ddrr(1)]&&Y\ar@{-}[dl(1)]&&B\ar@{-}[ddll(1)]&&+&&A\ar@{-}[ddrr(1)]&&X\ar@{-}[dr(1)]&&Y.\ar@{-}[ddll(1)] \\
&&&&&X\ar@{-}[dr(1)]&&Y\ar@{-}[dl(1)]&&&&&&&&&&&&&&&& \\
&&&&\text{\footnotesize $(a)$}&&&&&&&\text{\footnotesize $(b)+(d)$}&&&&&&&&\text{\footnotesize $(c)+(e)$}&&&& 
}$$
\hfill $\square$
\end{remark}
We can generalize these results for more than two derivations.
\begin{lemma}
For $k\geq 3$ and $A_1, \dots, A_k\in \mathbb{L}(V)$, we have $[[[[D, \sigma_{A_1}], \sigma_{A_2}]\cdots ], \sigma_{A_k}]=0$ on $\mathbb{L}(W_0)$.
\end{lemma}
\begin{proof}
By Lemma \ref{secondformulalemma} we have $[[D, \sigma_{A_1}], \sigma_{A_2}]=(-1)^{|A_1|}\sigma_{[A_1, A_2]}$ and by Lemma \ref{composition2derivations} (ii), $[\sigma_B, \sigma_C]=0$ for any $B, C\in \mathbb{L}(V)$. Therefore, $[[[ D, \sigma_{A_1}], \sigma_{A_2}], \sigma_{A_3}]=-(-1)^{|A_1|}[\sigma_{[A_1, A_2]}, \sigma_{A_3}]=0$ and the result follows.
\end{proof}
\begin{lemma}\label{3derivations}
For $A, B, C\in \mathbb{L}(V_0)$ and $T\in \mathbb{L}(W_0)$, we have
\begin{align*}
D(\sigma_A \sigma_B \sigma_C (T)))&=(-1)^{|A|+|B|}\sigma_A\sigma_{[B, C]}(T)-(-1)^{|A||B|}\sigma_B \sigma_{[A, C]}(T)\\
&-(-1)^{|A|}\sigma_{[A, B]}\sigma_C (T)-(-1)^{|A|+|B|+|C||T|}[\sigma_A\sigma_B (T), C]\\
&+[A, \sigma_B\sigma_C (T)]-(-1)^{|A|+|B|+|A||B|}[B, \sigma_A\sigma_C (T)].
\end{align*}
\end{lemma}
\begin{proof}
This follows from the previous Lemma and the development of the bracket $[[[D, \sigma_A ],\sigma_B], \sigma_C]$.
\end{proof}

\subsection{Star operator}

Let $V$ be a finite dimensional graded $\mathbb{Q}$-vector space. We will use the following description of $\mathbb{L}(V)$. For details on this construction see \cite[Section 8.1]{BF}.

First we consider the free graded linear magma on $V$ (graded vector space with a binary operation), $\mathbb{M}(V)=\bigoplus_{n=1}^\infty M_n$ defined by
$$M_1 = V\ \text{ and } M_n= \bigoplus _{i+j=n} M_i \otimes M_j\ \text{ for } n\geq 2,$$
with the (not associative) multiplication induced by the inclusions $M_i \otimes M_j \hookrightarrow M_{i+j}$. For $a \in M_i $ and $b\in M_j$, we write $ab$ instead of $a\otimes b$. The degree of the elements of $V$ induces naturally a degree on $\mathbb{M}(V)$. Let $I$ be the two-sided ideal of $\mathbb{M}(V)$ generated by all the elements of the form
\begin{equation}\label{FreeLie}
ab + (-1)^{|a||b|}ba,\hskip .5cm J(a, b, c)=a(bc)-(ab)c-(-1)^{|a||b|}b(ac).
\end{equation}

Denote by $\underline{A}$ the image of $A\in \mathbb{M}(V)$ under the natural projection in the quotient vector space $\mathbb{M}(V)\to \mathbb{M}(V)/I$. The bracket $[\underline{A}, \underline{B}]=\underline{AB}$ is well defined and endows $\mathbb{M}(V)/I$ with the structure of graded Lie algebra.

Moreover, any linear map $f\colon V\to L$ of degree $0$, where $L$ is a graded Lie algebra can be uniquely extended to $\mathbb{M}(V)$ through the formula $f(AB)=[f(A), f(B)]$ and this extension uniquely factorizes through $\mathbb{M}(V)/I$ giving rise to a morphism of graded Lie algebras.

This exhibits $\mathbb{M}(V)/I = \mathbb{L}(V)$ as the free graded Lie algebra on $V$. 

\begin{definition}\label{staroperator}
The {\em star operator} is the linear map 
$$\star \colon \mathbb{M}(V)\otimes \mathbb{L}(W)\to \mathbb{L}(V\oplus W \oplus s(V\otimes W)),$$ of degree $2$ defined inductively by $a\star T=0$ if $a\in V$ and $T\in \mathbb{L}(W)$ and
\begin{equation} \label{definitionstaroperator}
AB\star T=(-1)^{|A|} \sigma_{\underline{A}}\sigma_{\underline{B}}(T)+(-1)^{|B||T|}[A\star T, \underline{B}]+[\underline{A}, B\star T],
\end{equation} \label{staroperator}
for $A, B\in \mathbb{M}(V)$ and $T\in \mathbb{L}(W)$.
\end{definition}
\begin{remark} ({\bf Star operator in tree notation}) As in previous remarks we also show how to represent the star product in tree notation. We denote the multiplication of two elements  $a,b \in \mathbb{M}(V)$ by $\xymatrixcolsep{.5pc}
\xymatrixrowsep{-.2pc}
\entrymodifiers={=<1pc>} \xymatrix{\hskip .7cm a\hskip .5cm b&\\ \hskip .6cm\RS{{LLL}{RRR}}&\hskip -.55cm\RS{{LLL}{RRR}}} $. Then, formula $(\ref{staroperator})$ becomes:

$$
 \xymatrixcolsep{.5pc}
\xymatrixrowsep{.5pc}
\entrymodifiers={=<1pc>} \xymatrix{
&&&&&&\text{ \LARGE $\underline{A}$}\ar@{-}[ddr(0.8)]&&\hskip -.2cm\text{ \LARGE $\underline{B}$}\ar@{-}[ddl(0.8)]&&&&&&&&& \\
\text{ \huge $A$\ }\ar@{=}[dr(1)]&&\text{ \huge$B$}\ar@{=}[dl(1)]&\text{ \LARGE $\star $}&\text{ \huge$T$}&=&\hskip -.2cm\pm&&\hskip .8cm\pm&&\text{\LARGE $A\star T$}\ar@{-}[dr(1)]&&\text{ \LARGE $\underline{B}$}\ar@{-}[dl(1)]&\hskip .2cm+&\text{ \LARGE $\underline{A}$}\ar@{-}[dr(1)]&&\text{ \LARGE $B\star T$}\ar@{-}[dl(1)]& \\
&&&&&&&\text{\huge $T$}&&&&&&&&&& 
}$$
\vskip .3cm
\end{remark}

\begin{remark}\label{antisymmetry}({\bf Star operator and antisymmetry})
The star operator respects antisymmetry. Indeed,
\begin{align*}
BA\star T&=(-1)^{|B|} \sigma_{\underline{B}}\sigma_{\underline{A}}(T)+(-1)^{|A||T|}[B\star T, \underline{A}]+[\underline{B}, A\star T]\\
&=(-1)^{|B|}(-1)^{(|A|+1)(|B|+1)}\sigma_{\underline{A}}\sigma_{\underline{B}}(T)\\
&\hskip .5cm -(-1)^{|B|(|A|+|T|)}[A\star T, \underline{B}]-(-1)^{|A||T|+|A|(|B|+|T|)}[\underline{A}, B\star T],
\end{align*}
where we have used Lemma \ref{composition2derivations}(ii) in the second equality. Therefore,
\begin{align*}
-(-1)^{|A||B|} BA\star T &=(-1)^{|A|} \sigma_{\underline{A}}\sigma_{\underline{B}}(T)+(-1)^{|B||T|}[A\star T, \underline{B}]+[\underline{A}, B\star T]\\
&=AB\star T.
\end{align*}
\end{remark}
\begin{example}
Consider the graded vector spaces $V=\langle a, b, c \rangle$ and $W=\langle y , z\rangle $, where $|a|=2$, $|b|=|y|=3$, $|c|=|z|=7$. Then,
\begin{align*}
 \mathbb{M}(a, b, c )&=\langle a, b, c \rangle \oplus \langle aa, ab, ac, ba, bb, bc, ca, cb, cc \rangle \\
 &\oplus \langle a(aa), (aa)a, a(ab), (ab)a, a(ac), (ac)a, \dots  \\
 &\hskip .6cm b(aa), (aa)b, b(ab), (ab)b, b(ac), (ac)b, \dots \\
 &\hskip .6cm c(aa), (aa)c, c(ab), (ab)c, c(ac), (ac)c, \dots \rangle \oplus \cdots 
\end{align*}
 The generators of $\mathbb{L}(V\oplus W \oplus s(V\otimes W))$ of the form $s(v\otimes w)$ where $v\in V$ and $w\in W$ will be denoted by $vw$ to avoid excessive notation and recall that $|vw|=|v|+|w|+1$. Consider $A=a, B=bc \in \mathbb{M}(V)$ and $T=[y , z]\in \mathbb{L}(W)$. Then we have,
\begin{align*}
AB\star T&=a(bc)\star [y, z]=\sigma_a\sigma_{[b, c]}([y,z])+\Bigl[ a \star [y, z], [b,c] \Bigr] +\Bigl[ a , bc\star [y, z]\Bigr]\\
&=\sigma_a \Bigl[ \sigma_{[b, c]}(y), z\Bigr] -\sigma_a \Bigl[ y, \sigma_{[b, c]}(z)\Bigr] - \Bigl[ a, \sigma_b\sigma_c[y, z]\Bigr]\\
&=\sigma_a \Bigl[ \sigma_y [b, c], z\Bigr] -\sigma_a \Bigl[ y , \sigma_z[b, c]\Bigr] - \Bigl[ a , \sigma_b [cy, z]\Bigr] - \Bigl[ a , \sigma_b [y, cz]\Bigr]\\
&=-\sigma_a \Bigl[ [by,c],z\Bigr]-\sigma_a \Bigl[ [b, cy], z\Bigr] + \sigma_a\Bigl[y, [bz, c]\Bigr]+\sigma_a\Bigl[ y , [b , cz]\Bigr]\\
&-\Bigl[a, [cy, bz]\Bigr] -\Bigl[ a, [by , cz]\Bigr]\\
&=-\Bigl[[by, c], az\Bigr] -\Bigl[[b, cy ], az\Bigr] +\Bigl[ ay , [bz, c]\Bigr] + \Bigl[ ay , [b, cz]\Bigr]\\
&-\Bigl[ a, [cy, bz]\Bigr] -\Bigl[ a, [by , cz]\Bigr].
\end{align*}

Then, a similar computation using Definition \ref{staroperator} leads us to:
\begin{eqnarray*}
-(ab)c\star [y, z]&=&-\Bigl[ cy ,[az, b]\Bigr] + \Bigl[ cy, [a, bz] \Bigr]-\Bigl[ [ay, b], cz\Bigr]\\
&& + \Bigl[ [a, by], cz\Bigr]+\Bigl[[by, az], c\Bigr] - \Bigl[ [ay, bz ], c\Bigr],
\end{eqnarray*}
\begin{eqnarray*}
-b(ac)\star[y, z]&=&-\Bigl[ [ay, c], bz\Bigr] + \Bigl[ [a, cy ], bz \Bigr]-\Bigl[ by, [az, c]\Bigr] \\
&&+ \Bigl[ by, [a, cz]\Bigr]+\Bigl[ b, [cy, az]\Bigr] - \Bigl[ b, [ay, cz ]\Bigr],
\end{eqnarray*}
and by Jacobi identity on the Lie algebra $\mathbb{L}(V\oplus W \oplus s(V\otimes W))$ we obtain
$$J(a, b, c)\star [y, z]=\Bigl( a(bc)- (ab)c-b(ac)\Bigr) \star [y, z]=0.$$
One might think that the star product also respects the Jacobi identity but the following example shows that this is not the case.
\end{example}
\begin{example} \label{ExampleStarProduct}
Consider the graded vector spaces $V=\langle a, b, c \rangle$ and $W'=\langle x, y , z\rangle $, where $|a|=|x|=2$, $|b|=|y|=3$, $|c|=|z|=7$. Then,
\begin{align*}
a(bc)\star \Bigl[ x, [y,z]\Bigr]&=\sigma_a\sigma_{[b,c]}([x ,[y, z]])+\Bigl[ a \star [x,[y,z]], [b, c] \Bigr] +\Bigl[ a, bc\star [x, [y,z]]\Bigr]\\
&=\sigma_a\Bigl(\Bigl[ \sigma_{[b, c]}(x), [y,z]\Bigr] +  \Bigl[ x, [\sigma_{[b,c]}(y), z]\Bigr]-\Bigl[ x, [y, \sigma_{[b,c]}(z)]\Bigr]\Bigr)\\
&-\Bigl[ a , \sigma_b \sigma_c [x,[y,z]]\Bigr]\\
&=\sigma_a\Bigl(\Bigl[ \sigma_x [b,c], [y,z]\Bigr] +  \Bigl[ x, [\sigma_y [b,c], z]\Bigr]-\Bigl[ x, [y, \sigma_z[b,c]]\Bigr]\Bigr)\\
&-\Bigl[ a, \sigma_b \Bigl( [cx, [y, z]]+[x,[cy, z]]+[x, [y, cz]]\Bigr) \Bigr]
\end{align*}
\begin{align*}
&=\sigma_a \Bigl( \Bigr[[bx, c]- [b , cx], [y, z]\Bigr]+\Bigr[ x, [-[by, c]-[b, cy], z]\Bigr]+\Bigl[x, [y, [bz, c] + [b, cz]]\Bigr]\Bigr) \\
&-\Bigl[a, [cx, [by, z]]+[cx, [y, bz]]\Bigr]-\Bigl[a, [bx, [cy, z]]+[x,[cy, bz]]\Bigr]\\
&-\Bigl[a, [bx,[y, cz]]+[x,[by, cz]]\Bigr]\\
&=-\Bigl[ [bx, c],[ay, z]\Bigr] +\Bigl[ [bx, c], [y, az]\Bigr]+\Bigl[[b, cx], [ay, z]\Bigr]\\
&\hskip .3cm -\Bigl[[b, cx], [y, az]\Bigr]-\Bigl[ ax, [[by, c], z]\Bigr]-\Bigl[x, [[by , c], az]\Bigr]\\
&\hskip .3cm -\Bigl[ax, [[b , cy], z]\Bigr]-\Bigl[ x, [[b , cy ], az]\Bigr] \hskip.08cm +\hskip .08cm\Bigr[ax, [y,[bz, c]]\Bigr]\\
&\hskip .45cm +\Bigr[x, [ay, [bz, c]]\Bigr]+\Bigl[ax, [y, [b , cz]]\Bigr] +\Bigl[x, [ay, [b, cz]]\Bigr]\\
&\hskip .3cm -\Bigl[ a, [cx, [by, z]]\Bigr] -\Bigl[ a, [cx, [y, bz]]\Bigr]-\Bigl[a, [bx, [cy, z]]\Bigr]\\
&\hskip .45cm - \Bigr[a, [x, [cy, bz]]\Bigr] - \Bigl[a, [bx, [y, cz]]\Bigr] - \Bigl[ a, [x, [by, cz]]\Bigr].
\end{align*}
Then, a similar computation using Definition \ref{staroperator} leads us to:
\begin{align*}
-(ab)c\star [x, [y, z]]
&=-\Bigl[ cx, [[ay, b], z]\Bigr]+\Bigl[cx, [[a, by], z]\Bigr]-\Bigl[cx, [y,[az, b]]\Bigr]\\
&\hskip .3cm+\Bigl[cx, [y, [a, bz]]\Bigr]+\Bigl[[ax, b], [cy, z]\Bigr]+\Bigl[[a, bx], [cy, z]\Bigr]\\
&\hskip .3cm-\Bigl[x, [cy, [az, b]]\Bigr]+\Bigl[x, [cy, [a, bz]]\Bigr]+\Bigl[[ax,b], [y, cz]\Bigr]\\
&\hskip .3cm+\Bigl[[a, bx], [y, cz]\Bigr]-\Bigl[x, [[ay, b], cz]\Bigr]+\Bigl[x, [[a, by], cz]\Bigr]\\
&\hskip .3cm-\Bigl[[bx, [ay, z]], c\Bigr] + \Bigl[[bx, [y, az]], c\Bigr]-\Bigl[[ax, [by, z]], c\Bigr]\\
&\hskip .3cm+\Bigl[[x, [by, az]], c\Bigr]-\Bigl[[ax, [y, bz]], c\Bigr]-\Bigl[[x, [ay, bz]], c\Bigr].
\end{align*}
\begin{align*}
-b(ac)\star [x, [y, z]]&=+\Bigl[ [ax, c],[by, z]\Bigr] +\Bigl[ [ax, c], [y, bz]\Bigr]+\Bigl[[a, cx], [by, z]\Bigr]\\
&\hskip .3cm +\Bigl[[a, cx], [y, bz]\Bigr]-\Bigl[ bx, [[ay, c], z]\Bigr]-\Bigl[x, [[ay , c], bz]\Bigr]\\
&\hskip .3cm +\Bigl[bx, [[a , cy], z]\Bigr]+\Bigl[ x, [[a , cy ], bz]\Bigr] \hskip.08cm -\hskip .08cm\Bigr[bx, [y,[az, c]]\Bigr]\\
&\hskip .45cm -\Bigr[x, [by, [az, c]]\Bigr]+\Bigl[bx, [y, [a , cz]]\Bigr] +\Bigl[x, [by, [a, cz]]\Bigr]\\
&\hskip .3cm -\Bigl[ b, [cx, [ay, z]]\Bigr] +\Bigl[ b, [cx, [y, az]]\Bigr]-\Bigl[b, [ax, [cy, z]]\Bigr]\\
&\hskip .45cm + \Bigr[b, [x, [cy, az]]\Bigr] - \Bigl[b, [ax, [y, cz]]\Bigr] - \Bigl[ b, [x, [ay, cz]]\Bigr] 
\end{align*}

We can compute now 
$$J(a, b, c)\star [x, [y,z]]=\Bigl( a(bc)- (ab)c-b(ac)\Bigr) \star [x, [y, z]].$$
For example, the terms including $ay$ and $bx$ are
$$-\Bigl[[bx, c], [ay, z]\Bigr] -\Bigl[[bx, [ay, z]],c\Bigr] - \Bigl[ bx, [[ay, c], z]\Bigr]=(\dag ).$$
Using Jacobi identity twice we obtain
\begin{align*}
(\dag )&=-\Bigl[[bx, c], [ay, z]\Bigr]
-\Bigl(\Bigl[ bx, [[ay, z], c]\Bigr]-\Bigl[[bx, c], [ay, z]\Bigr]\Bigr)-\Bigl[bx, [[ay, c], z]\Bigr]\\
&= -\Bigl[[bx, c], [ay, z]\Bigr]+ \Bigl[[bx, c], [ay, z]\Bigr]- \Bigl( \Bigl[ bx ,[ ay,[z, c]]\Bigr]-\Bigl[ bx, [[ay, c], z]\Bigr]\Bigr)\\
&\hskip .35cm -\Bigl[bx, [[ay, c], z]\Bigr]\\
&=-\Bigl[ bx, [ay , [z, c]]\Bigr]\not= 0.
\end{align*}

With a similar computation, the terms including $bx$ and $az$ are
$$\Bigl[[bx, c], [y, az]\Bigr]+\Bigl[ [bx, [y, az]], c\Bigr]- \Bigl[ bx , [y, [az, c]]\Bigr] = \Bigl[bx, [[c, y], az]\Bigr]\not= 0,$$
and so on with the $18$ different combinations obtaining
\begin{align*}
J(a, b, c)\star [x, [y,z]]&=\Bigl( a(bc)- (ab)c-b(ac)\Bigr) * [x, [y, z]]\\
&=+\Bigl[ [c, x], [by, az]\Bigr]+\Bigl[ cx, [[b, y], az\Bigr]-\Bigl[ cx, [by, [a, z]]]\Bigr]\\
&\hskip .3cm-\Bigl[[c, x],[ay, bz]]\Bigr] -\Bigl[cx,[[a, y], bz]\Bigr]-\Bigl[cx, [ay, [b, z]]\Bigr]\\
&\hskip .3cm+\Bigl[[b, x], [cy, az]\Bigr]+ \Bigl[bx, [[c, y], az]\Bigr]-\Bigl[ bx, [cy, [a, z]]\Bigr]\\
&\hskip .3cm-\Bigl[[a, x], [cy, bz]\Bigr] +\Bigl[ax, [[c, y], bz]\Bigr]-\Bigl[ax, [cy, [b, z]]\Bigr]\\
&\hskip .3cm - \Bigl[ [b , x], [ay, cz]\Bigr] -\Bigl[ bx, [[a, y], cz]\Bigr] -\Bigl[ bx, [ay, [c, z]]\Bigr]\\
&\hskip .3cm-\Bigl[[a, x], [by, cz]\Bigr] + \Bigl[ ax, [[b, y], cz]\Bigr] -\Bigl[ ax, [by, [c ,z]]\Bigr].
\end{align*}
One can check that if $a, b, c, \in V_0$ and $x, y , z \in W_0$ the above summation agrees with
\begin{align*}
-D\Bigl(\sigma_a \sigma_b\sigma_c ([x, [y, z]])\Bigr)&=-D\Bigl( -\Bigl[ cx, [ by , az]\Bigr] + \Bigl[ cx, [ay, bz]\Bigr]-\Bigl[ bx, [cy, az]\Bigr]  \\
&\hskip 0.8cm +\Bigl[ ax, [cy, bz]\Bigr] +\Bigl[ bx, [ay, cz]\Bigr] +\Bigl[ ax, [by, cz]\Bigr] \Bigr).
\end{align*}
This is not a mere coincidence but a particular example of the next Lemma.
\end{example}

\begin{lemma}
For $A, B, C\in \mathbb{M}(V_0)$ and $T\in \mathbb{L}(W_0)$, we have
$$J(A, B, C)\star T=(-1)^{|B|}D(\sigma_{\underline{A}}\sigma_{\underline{B}}\sigma_{\underline{C}} (T)).$$
\end{lemma}

\begin{proof}
We can compute $J(A, B, C)\star T = \Bigl( A (BC) - (AB)C - (-1)^{|A||B|}B(AC)\Bigr)\star T$ using the recursive definition of the star operator obtaining

\begin{align*}
&J(A, B, C)\star T=\\
&=(-1)^{|A|}\sigma_{\underline{A}}\sigma_{[\underline{B}, \underline{C}]}(T)+(-1)^{(|A|+|B|)|C|+|C|}\sigma_{\underline{C}}\sigma_{[\underline{A}, \underline{B}]}(T)\\
&+(-1)^{(|A||B|}\sigma_{\underline{B}}\sigma_{[\underline{A}, \underline{C}]}(T)+(-1)^{|B|}[\underline{A}, \sigma_{\underline{B}}\sigma_{\underline{C}}(T)]\\
&+(-1)^{(|A|+|B|)|C|+|A|}[\underline{C}, \sigma_{\underline{A}}\sigma_{\underline{B}}(T)]-(-1)^{|A||B|+|A|}[\underline{B}, \sigma_{\underline{A}}\sigma_{\underline{C}}(T)]\\
&+(-1)^{(|B|+|C|)|T|} J(A\star T, \underline{B}, \underline{C})+(-1)^{|C||T|}J(\underline{A}, B\star T, \underline{C})+ J(\underline{A}, \underline{B}, C\star T).
\end{align*}
\vskip .3cm
Since this expression lies in the Lie algebra $\mathbb{L}(V\oplus W\oplus s(V\otimes W))$, the last line vanishes and the result follows by comparison with Lemma \ref{3derivations}.
\end{proof}

\begin{lemma} \label{LemmaDstar}
Let $A\in \mathbb{L}(V_0)$ and $\overline{A}\in \mathbb{M}(V_0)$ any representative of $A$. Let $T\in \mathbb{L}(W_0)$. Then we have
\begin{equation}\label{Dstar}
D(\overline{A}\star T)=-\sigma_A (T)+(-1)^{|A||T|}\sigma_T(A).
\end{equation} 
\end{lemma}

\begin{proof}
We proceed by induction on the bracket-length of $A\in \mathbb{L}(V_0)$. The statement is clearly true if $A=v\in V_0$. Indeed, $D(v\star T )=D(0)=0$ by definition of the star product and $-\sigma_v (T)+ (-1)^{|v||T|}\sigma_T (v)=0$ by definition of the derivation $\sigma_T$.

Suppose that the statement is true for $A, B\in \mathbb{L}(V_0)$ and consider $[A, B]\in \mathbb{L}(V_0)$. Note that the product $\overline{A}\overline{B}\in \mathbb{M}(V_0)$ is a representative of the bracket $[A, B]\in \mathbb{L}(V_0)$. For $T\in \mathbb{L}(W_0)$ we have
\begin{align*}
D(\overline{A}\overline{B}\star T )&=D\Bigl( (-1)^{|A|}\sigma_A\sigma_B (T) +(-1)^{|B||T|}[\overline{A}\star T, B ]+[A, \overline{B}\star T]\Bigr)\\
&=(-1)^{|A|}D\Bigl( \sigma_A \sigma_B (T)\Bigr)\\
&-(-1)^{|B||T|}[\sigma_A (T), B]+(-1)^{(|A|+|B|)|T|}[\sigma_T (A), B]\\
&-(-1)^{|A|}[A, \sigma_B (T)]+(-1)^{|A|+|B||T|}[A, \sigma_T(B)].
\end{align*}
Using formula ($\ref{secondformula}$) of Lemma \ref{secondformulalemma} we obtain
\begin{align*}
D(\overline{A}\overline{B}\star T )&=-\sigma_{[A, B]}(T)+(-1)^{(|A|+|B|)|T|}[\sigma_T(A), B]+(-1)^{|A|+|B||T|}[A, \sigma_T (B)]\\
&=-\sigma_{[A, B]}(T)+(-1)^{(|A|+|B|)|T|}\sigma_T[A, B],
\end{align*}
which completes the proof.
\end{proof}

\begin{remark}
As usual we can write Lemma \ref{LemmaDstar} in tree notation as:
\begin{equation*}
\xymatrixcolsep{.5pc}
\xymatrixrowsep{.5pc}
\entrymodifiers={=<1pc>} \xymatrix{
&&&A\ar@{-}[dd]&&T\ar@{-}[dd] & \\
\hskip -.8cm D( \overline{A}\star T)&=&-&&\pm && , \\
&&&T&&A&
}
\end{equation*}
\end{remark}
\subsection{The model of the product of $2$-cones}\label{ModelProduct}

The lemmas of the previous sections allow us to prove the following Theorem.

\begin{theorem} \label{modelof2cones}
Let $X$ and $Y$ be $2$-cones modeled by $(\mathbb{L}(V), \partial _V)$ and $(\mathbb{L}(W), \partial _W)$ respectively, where $V=V_0 \oplus V_1$, $W=W_0\oplus W_1$ with $\partial_V (V_0)=\partial _W (W_0)=0$ and $\partial_V (V_1)\subset \mathbb{L}(V_0)$, $\partial_W (W_1)\subset \mathbb{L}(W_0)$.
Then, the minimal Quillen model of the Cartesian product $X\times Y$ has the form 
$$\Bigl( \mathbb{L}(V\oplus W \oplus s(V\otimes W)), D \Bigr),$$
where $D(v)= \partial _V v$, $D(w)= \partial _W w$ and
\begin{equation}\label{Differential}
D( s(v\otimes w))=[v, w]-(-1)^{(|v|+1)|w|}\sigma_{w}(\partial_V v) -(-1)^{|v|}\sigma_v (\partial_W w)+(-1)^{|v|}\overline{\partial_V v}\star \partial_W w,
\end{equation}
for every $v\in V$, $w\in W$ and $\overline{\partial_V v}$ a representative of $\partial_V v$ in $\mathbb{M}(V)$.
\end{theorem}
\begin{proof}
In view of Lemma \ref{Quism} we just have to prove that $D$ is a differential. In order to do that, we have to check that $D^2 (s(v\otimes w))=0$. 

Since $\partial_V v \in \mathbb{L}(V_0)$ and $\partial_W w\in \mathbb{L}(W_0)$ we have:

\begin{align*}
D\Bigl(\sigma_w (\partial_Vv)\Bigr)&=(-1)^{|w|(|v|+1)}[\partial_Vv, w] -\sigma_{\partial_Ww}(\partial_Vv),&(\text{Remark \ref{iii and iv} (iv)})\\ 
D\Bigl(\sigma_v (\partial_Ww)\Bigr)&=[v, \partial_Ww] -\sigma_{\partial_Vv}(\partial_Ww),&  (\text{Remark \ref{iii and iv} (iii)})\\
D\Bigl(\overline{\partial_Vv}\star \partial_Ww\Bigr)&=-\sigma_{\partial_Vv}(\partial_Ww)+(-1)^{(|v|+1)(|w|+1)}\sigma_{\partial_Ww}(\partial_Vv),&  (\text{Lemma \ref{LemmaDstar}}).
\end{align*} 
Therefore,
\begin{align*}
D^2 (s(v\otimes w))&=D\Bigl( [v, w]-(-1)^{(|v|+1)|w|}\sigma_{w}(\partial_V v) -(-1)^{|v|}\sigma_v (\partial_W w)\\
&+(-1)^{|v|}\overline{\partial_V v}\star \partial_W w \Bigr)\\
&=[\partial_V v , w ] +(-1)^{|v|}[v, \partial_Ww]\\
&-[\partial_Vv, w]+(-1)^{(|v|+1)|w|}\sigma_{\partial_Ww}(\partial_Vv)\\
&-(-1)^{|v|}[v, \partial_Ww]+(-1)^{|v|}\sigma_{\partial_Vv}(\partial_Ww)\\
&-(-1)^{|v|}\sigma_{\partial_Vv}(\partial_Ww)-(-1)^{(|v|+1)|w|}\sigma_{\partial_ww}(\partial_Vv)=0.
\end{align*}
\end{proof}

\begin{example} \label{exampleproduct}
We will give a complete example of the model of the product of two $2$-cones using formula ($\ref{Differential}$).

Let $X$ be the space whose Quillen minimal model is of the form $(\mathbb{L}(V), \partial ) = (\mathbb{L}(a, b, c, v), \partial)$, where $|a|= 2$, $|b|= 3$, $|c|=7$, $|v|=13$ and the differential is given by $\partial a = \partial b =\partial c =0$ and $\partial v =[a, [b,c]]$. Then, $X$ is a $2$-cone with $V=V_0 \oplus V_1$ where $V_0=\langle a, b, c\rangle$ and $V_1=\langle v\rangle$.

Consider the Cartesian product $X\times X$. The model of the second factor will be denoted as $(\mathbb{L}(W), \partial ) = (\mathbb{L}(x, y, z, w), \partial)$, where $|x|= 2$, $|y|= 3$, $|z|=7$, $|w|=13$ and the differential given by $\partial x = \partial y =\partial z =0$ and $\partial w =[x, [y,z]]$. $W_0=\langle x, y, z\rangle$ and $W_1=\langle w\rangle$.

The minimal Quillen model of $X\times X$ is of the form
$$\Bigl( \mathbb{L}(V\oplus W \oplus s(V\otimes W)), D \Bigr).$$

Let us compute the differential of this model using Theorem \ref{modelof2cones}. We will suppress suspensions $s$ and tensor products $\otimes$ to avoid excessive notation. First

$$\begin{array}{lcccr} 
s(V_0\otimes W_0)=\langle &ax,&ay,&az,&\\
&bx,&by,&bz,&\\
&cx,&cy,&cz,&\rangle ,\\
\end{array}$$
$$\begin{array}{ccc}
D(ax)=[a, x]&D(ay)=[a, y]&D(az) =[a,z]\\
D(bx)=[b, x]&D(by)=[b, y]&D(bz) =[b,z]\\
D(cx)=[c, x]&D(cy)=[c, y]&D(az) =[c,z]
\end{array}$$

Second,
$$\begin{array}{lcccr} 
s(V_0\otimes W_1)\oplus s(V_1\otimes W_0)=\langle &aw,&bw,&cw,&\\
&vx,&vy,&vz,&\rangle .
\end{array}$$

Note that we cannot set $D(aw)=[a, w]$ since $D[a, w]=[a, [x, [y, z]]\not=0$. So we need to add the term $-(-1)^{|a|}\sigma_a([x,[y, z]])$: 
\begin{align*}
D(aw) &=[a, w]-(-1)^{|a|}\sigma_a (\partial w)\\
&=[a, w] - \sigma_a ([x, [y, z]])\\
&=[a, w] -[ax, [y, z]] -[x, [ay, z]]+[x, [y, az]].
\end{align*}
Indeed,
\begin{align*}
D(-\sigma_a([x, [y, z]])&=-\Bigl[[a, x], [y, z]\Bigr]-\Bigl[x, [[a, y], z]\Bigr]-\Bigl[x, [y, [a, z]]\Bigr],
\end{align*}
and by Jacobi identity
$$D[a, w]=\Bigl[a, [x, [y, z]]\Bigr]=\Bigl[[a, x], [y, z]\Bigr]+\Bigl[x, [[a, y], z]\Bigr]+\Bigl[x, [y, [a, z]]\Bigr],$$
obtaining $D^2 (aw)=0$. In the same way we have
\begin{align*}
D(bw)&=[b, w]-(-1)^{|b|}\sigma_b (\partial w)=[b, w] +[bx, [y, z]] +[x, [by, z]]+[x, [y, bz]],\\
D(cw)&=[c, w]-(-1)^{|c|}\sigma_c(\partial w)=[c, w] +[cx, [y, z]] +[x, [cy, z]]+[x, [y, cz]],\\
D(vx)&= [v, x]-(-1)^{(|v|+1)|x|}\sigma_x(\partial v)\\
&=[v, x] -[ax, [b, c]] -[a, [bx, c]]+[a, [b, cx]],\\
D(vy)&= [v, y]-(-1)^{(|v|+1)|y|}\sigma_y(\partial v)\\
&=[v, y] -[ay, [b, c]] +[a, [by, c]]+[a, [b, cy]],\\
D(vz)&= [v, z]-(-1)^{(|v|+1)|z|}\sigma_z(\partial v)\\
&=[v, z] -[az, [b, c]] +[a, [bz, c]]+[a, [b, cz]].
\end{align*}

Finally, $s(V_1 \otimes W_1)=\langle \ vw\  \rangle$.

By formula ($\ref{Differential}$) we have

$$D(vw)=[v, w]-(-1)^{(|v|+1)|w|}\sigma_w (\partial v)-(-1)^{|v|}\sigma_v (\partial w)+(-1)^{|v|}\overline{\partial v}\star\partial w,$$

where
\begin{align*}
\sigma_w(\partial v)&=\sigma_w ([a, [b, c]])=\Bigl[aw, [b, c]\Bigr]-\Bigl[a, [bw, c]\Bigr]-\Bigl[a, [b, cw]\Bigr],\\
\sigma_v (\partial w)&=\sigma_v ([x, [y, z]])=\Bigl[vx, [y, z]\Bigl]+\Bigl[x, [vy, z]\Bigr]+\Bigl[x, [y, vz]\Bigr],\\
\overline{\partial v}\star \partial w &=a(bc)\star [x, [y, z]].
\end{align*}

The expression for $a(bc)\star [x, [y, z]]$ has been computed in detail in Example \ref{ExampleStarProduct}. 

Writting all the terms together, with the appropriate sign we obtain:
\begin{align*}
D (vw)&=[v, w]\\
&-\Bigl[aw, [b, c]\Bigr]+\Bigl[a, [bw, c]\Bigr]+\Bigl[a, [b, cw]\Bigr]\\
&+\Bigl[vx, [y, z]\Bigl]+\Bigl[x, [vy, z]\Bigr]+\Bigl[x, [y, vz]\Bigr]\\
&+\Bigl[ [bx, c],[ay, z]\Bigr] -\Bigl[ [bx, c], [y, az]\Bigr]-\Bigl[[b, cx], [ay, z]\Bigr]\\
&+\Bigl[[b, cx], [y, az]\Bigr]+\Bigl[ ax, [[by, c], z]\Bigr]+\Bigl[x, [[by , c], az]\Bigr]\\
&+\Bigl[ax, [[b , cy], z]\Bigr]+\Bigl[ x, [[b , cy ], az]\Bigr] \hskip.08cm -\hskip .08cm\Bigr[ax, [y,[bz, c]]\Bigr]\\
&\hskip .1cm -\hskip .05cm\Bigr[x, [ay, [bz, c]]\Bigr]-\Bigl[ax, [y, [b , cz]]\Bigr] -\Bigl[x, [ay, [b, cz]]\Bigr]\\
&+\Bigl[ a, [cx, [by, z]]\Bigr] +\Bigl[ a, [cx, [y, bz]]\Bigr]+\Bigl[a, [bx, [cy, z]]\Bigr]\\
& \hskip .1cm +\hskip .05cm \Bigr[a, [x, [cy, bz]]\Bigr] + \Bigl[a, [bx, [y, cz]]\Bigr] + \Bigl[ a, [x, [by, cz]]\Bigr].
\end{align*}
\end{example}
\section{Applications and further Examples}

In this section we will give some applications of the models developed and some additional examples. 






\subsection{Model of the diagonal map}

Let $(\mathbb{L}(V), \partial )$ with $V=V_0\oplus V_1$, $\partial V_0 =0$, $\partial V_1 \subset \mathbb{L}(V_0)$ be the Quillen minimal model of a $2$-cone $X$. The goal of this section is to describe a model for the diagonal map $\Delta \colon X\to X\times X$ whose target is the model of the product $X\times X$ described in section \ref{ModelProduct}. Writing $V'$ for a copy of $V$, the model of $X\times X$ is given by $(L, D)=(\mathbb{L}(V\oplus V' \oplus s(V\otimes V')), D)$ with the differential describe in Theorem \ref{modelof2cones}.

On $\mathbb{L}(V_0)$ the diagonal map is clearly modelled by the Lie morphism
\begin{equation}\label{firstdiagonal}
\Delta \colon (\mathbb{L}(V_0), 0)\to (L, D)\hskip .2cm \text{given by}\hskip .2cm \Delta (v)=v+v'.
\end{equation}

In order to extend this morphism we first define $\Gamma \colon \mathbb{M}(V_0)\to L$ inductively by $\Gamma v =0$ for $v\in V_0$ and
\begin{equation}\label{diagonal}
\Gamma (AB) =\sigma_{\underline{A}}(\underline{B}')-(-1)^{|A||B|}\sigma_{\underline{B}}(\underline{A}')+\frac{1}{2}(-1)^{|A|}[\Delta \underline{A}+\underline{A}+\underline{A}', \Gamma B]+\frac{1}{2}[\Gamma A, \Delta \underline{B}+\underline{B}+\underline{B}']
\end{equation}
for $A, B\in \mathbb{M}(V_0)$.
\begin{lemma}\label{differentialGamma}
For any $A\in \mathbb{M}(V_0)$, we have $D\Gamma A = \Delta \underline{A}-\underline{A}-\underline{A}'$.
\end{lemma}
\begin{proof}
We proceed by induction on the length of $A$. For $A=v\in V_0$, we have $\underline{v}=v$ and the result is immediate since $\Delta (v)=v+v'$ and $\Gamma v=0$. Assume that the equality holds for $A, B\in \mathbb{M}(V_0)$.  Then, since $D\Gamma A = \Delta \underline{A}-\underline{A}-\underline{A}'$ and $D\Gamma B = \Delta \underline{B}-\underline{B}-\underline{B}'$, we have that $D\Gamma A + 2 (\underline{A}+\underline{A}')= \Delta A +\underline{A}+\underline{A}'$ and $D\Gamma B + 2 (\underline{B}+\underline{B}')= \Delta B +\underline{B}+\underline{B}'$.

 The calculation of $D\Gamma (AB)$ gives
\begin{align*}
D\Gamma (AB)&=[\underline{A}, \underline{B}']-(-1)^{|A||B|}[\underline{B}, \underline{A}'] \hskip 3.4cm (\text{by Remark \ref{iii and iv}})\\
&+\frac{1}{2}[\Delta \underline{A}+\underline{A}+\underline{A}', \Delta \underline{B}-\underline{B}-\underline{B}' ] +\frac{1}{2}[\Delta \underline{A}-\underline{A}-\underline{A}', \Delta \underline{B}+\underline{B}+\underline{B}' ]\\
&=[\underline{A}, \underline{B}']-(-1)^{|A||B|}[\underline{B}, \underline{A}']+[\Delta \underline{A}, \Delta \underline{B}]-[\underline{A}+\underline{A}', \underline{B}+\underline{B}']\\
&=\Delta [\underline{A}, \underline{B}] -[\underline{A}, \underline{B}]-[\underline{A}', \underline{B}']\\
&= \Delta \underline{AB}- \underline{AB}-\underline{AB}',
\end{align*}
since $\underline{AB}=[\underline{A}, \underline{B}]$.
\end{proof}

As the star operator, $\Gamma$ is compatible with antisymmetry but not with Jacobi identity, i.e. $\Gamma J(A, B, C)= \Gamma \Bigl( A (BC) - (AB)C + (-1)^{|A||B|}B (AC)\Bigr)$ does not vanish necessarily. However, we have:
\begin{lemma} \label{JacobiGamma}
For any $A, B, C \in \mathbb{M}(V_0)$, $\Gamma J (A, B, C)$ is a boundary.
\end{lemma}
\begin{proof}
We can check that $\Gamma J (A, B, C)$ is the boundary of the element
\begin{align*}
&(-1)^{|A|}\sigma_{\underline{A}}\sigma_{\underline{B}}(\underline{C}')-(-1)^{|A|+|B||C|}\sigma_{\underline{A}}\sigma_{\underline{C}}(\underline{B}')+(-1)^{|A||B|+|A||C|+|B|}\sigma_{\underline{B}}\sigma_{\underline{C}}(\underline{A}')\\
+&\frac{1}{2}(-1)^{|A|}\Bigr[\Gamma A, \sigma_{\underline{B}}(\underline{C}')-(-1)^{|B||C|}\sigma_{\underline{C}}(\underline{B}')-\frac{1}{6}(-1)^{|B|}[D\Gamma B, \Gamma C]-\frac{1}{6}[\Gamma B, D\Gamma C]\Bigl]\\
+&\frac{1}{2}(-1)^{|A|+|B|}\Bigl[ \sigma_{\underline{A}}(\underline{B}')-(-1)^{|A||B|}\sigma_{\underline{B}}(\underline{A}')-\frac{1}{6}(-1)^{|A|}[D\Gamma A, \Gamma B ]-\frac{1}{6}[\Gamma A, D\Gamma B], \Gamma C \Bigr]\\
-&\frac{1}{2}(-1)^{|A||B|+|B|}\Bigl[ \Gamma B, \sigma_{\underline{A}}(\underline{C}')-(-1)^{|A||C|}\sigma_{\underline{C}}(\underline{A}')-\frac{1}{6}(-1)^{|A|}[D\Gamma A, \Gamma C]-\frac{1}{6}[\Gamma A, D \Gamma C]\Bigr]
\end{align*}
\end{proof}
Then, we can state:
\begin{theorem}\label{diagonalmodel}
Any Lie morphism $\Delta \colon (\mathbb{L}(V), \partial )\to (L, D)$ extending the morphism given in $(\ref{firstdiagonal})$ by
\begin{equation}\label{thediagonalmodel}
\Delta(v)=v+v' +\Gamma \overline{\partial v}\ \text{ for } v\in V_1
\end{equation}
where $\overline{\partial v}$ is a representative of $\partial v$ in $\mathbb{M}(V_0)$, is a model of the diagonal map $\Delta\colon X\to X\times X$.
\end{theorem}
\begin{remark}
Note that by Lemma \ref{JacobiGamma}, any two such morphisms differ from at most a boundary.
\end{remark}
\begin{proof}[Proof of Theorem \ref{diagonalmodel}]
The Lie morphism $\Delta $ commutes with the differentials. Indeed
\begin{align*}
D\Delta (v)&=D(v+v'+\Gamma \overline{\partial v})\\
&=(\partial v)+(\partial v)' + D\Gamma \overline{\partial v}\\
&=(\partial v)+(\partial v)' + \Delta (\partial v ) -(\partial v)-(\partial v)' \hskip 2cm (\text{by Lemma}\ \ref{differentialGamma})\\
&=\Delta (\partial v).
\end{align*}
Then, $\Delta $ is a model for the diagonal map $\Delta \colon X\to X\times X$ since the image of $\Gamma$ is contained in the Lie ideal of $L$ generated by $s(V\otimes V')$, which implies that the composition of $\Delta$ with both projections $(L, D)\to (\mathbb{L}(V), \partial )$ and  $(L, D)\to (\mathbb{L}(V'), \partial' )$ is the identity.
\end{proof}

\begin{example}
Consider the space $X$ of Example \ref{exampleproduct} whose minimal Quillen model is of the form $(\mathbb{L}(V_0\oplus V_1), \partial )$ with $V_0=\langle a, b, c \rangle$ and $V_1=\langle v \rangle$ whith $|a|=|b|=|c|=2$, $|v|=7$ and $\partial v =[a, [b, c]]$. As in Example \ref{exampleproduct} we take a copy of this model of the form $(\mathbb{L}(W_0\oplus W_1), \partial )$ with $W_0=\langle x, y, z \rangle$ and $W_1=\langle w \rangle$ whith $|x|=|y|=|z|=2$, $|w|=7$ and $\partial w =[x, [y, z]]$.

The model of $X\times X$ is $L=( \mathbb{L}(V\oplus W\oplus s(V\otimes W)), D )$ (see the aforementioned example for the explicit differential).

We will compute formula $(\ref{thediagonalmodel})$ to obtain a model 
$$\Delta \colon (\mathbb{L}(V), \partial )\to \Bigr( \mathbb{L}(V\oplus W\oplus s(V\otimes W)), D \Bigl)$$
of the diagonal map $X\to X\times X$.

First, since $\partial a =\partial b =\partial c=0$, we have $\Delta a = a + x$, $\Delta b = b + y$ and $\Delta c = c+ z$. It only remains to calculate $\Delta  v$. As $\partial v =[a, [b, c]]$ and $\Delta $ is a Lie morphism we have 
\begin{align*}
\Delta (\partial v )&=\Delta ([a, [b, c]])=[\Delta a, [\Delta b , \Delta c ]]\\
&=[a, [b, c]]+ [a, [b, z]]+[a, [y, c]]+[a, [y, z]]\\
&+[x, [b, c]]+[x, [b, z]]+[x, [y, c]]+[x, [y, z]]. 
\end{align*}
However, if we define $\Delta v = v +w$ we will obtain only the first and last terms of the above summation, i.e. $D (\Delta v )=\partial v +\partial w =[a, [b, c]]+[x, [y, z]]$. But formula $(\ref{thediagonalmodel})$ $\Delta v = v + w +\Gamma (\overline{\partial v })$ fix the problem.

Since $\partial v =[a, [b, c]]$ we can write $\overline{\partial v}= a (bc)$. Then,
\begin{align*}
\Gamma \Bigl( a (bc)\Bigr)&= \sigma _a ([y, z])-\sigma_{[b, c]}(x) +\frac{1}{2}[\Delta a + a+ x, \Gamma (bc)]\\
&=[ay , z]-[y, az]-[bx, c]+[b, cx]\\
&+\frac{1}{2}[2(a+x), \sigma_b (z)-\sigma_c (y)]\\
&=[ay , z]-[y, az]-[bx, c]+[b, cx]\\
&+[a, bz]+[a, cy]+[x, bz]+[x, cy].
\end{align*}
Then, we have
\begin{align*}
\Delta v=v + w&+ [ay , z]-[y, az]-[bx, c]+[b, cx]\\
&+[a, bz]+[a, cy]+[x, bz]+[x, cy].
\end{align*}
One can check that $D(\Delta v)=\Delta (\partial v)$ using Jacobi identity and antisymmetry.
\end{example}

\subsection{An example of a product of $3$-cones}

We will finish with a last example of a product of two 3-cones whose minimal Quillen model is not entirely described by the differential given by formula $(\ref{Differential})$. The space considered will be the same as in Example \ref{exampleproduct} but with a slight modification of its differential.
\begin{example}
Let $Y$ be the space whose Quillen minimal model is given by 
\begin{equation*}
(\mathbb{L}(V),\partial )=(\mathbb{L}(V_0\oplus V_1\oplus V_2), \partial ),
\end{equation*}
where $V_0=\langle a, b\rangle $ with $|a|=2$, $|b|=3$, $V_1=\langle c\rangle$ with $\partial c=[b,b]$ and $V_2=\langle v\rangle $ with $\partial v=[a,[b,c]]$. In other terms, $Y$ is a $3$-cone.
We will compute explicitely the Quillen model of $Y\times Y$.
 We will denote by 
\begin{equation*}
(\mathbb{L}(W),\partial_W )=(\mathbb{L}(W_0\oplus W_1\oplus W_2), \partial ),
\end{equation*}
where $W_0=\langle x, y\rangle $ with $|x|=2$, $|y|=3$, $W_1=\langle z\rangle$ with $\partial z=[y,y]$ and $W_2=\langle w\rangle $ with $\partial w=[x,[y,z]]$, the second factor of the product.

The minimal Quillen model of the product $Y\times Y$ has the form
\begin{equation*}
\Bigl( \mathbb{L}(V\oplus W \oplus s(V\otimes W)), D \Bigr)
\end{equation*}
although we will suppress suspensions $s$ and tensor products $\otimes$ to avoid excessive notation. Then 
$$\begin{array}{lccccr}
s(V\otimes W)=\langle &ax,&ay,&az,&aw,&\\
&bx,&by,&bz,&bw,&\\
&cx,&cy,&cz,&cw,& \\
 &vx,&vy,&vz,&vw&\rangle .\\ 
\end{array}$$
\vskip .2cm

We will compute the differential on the generators of $s(V\otimes W)$. Note that the only generators with cone-length $3$ on the model of $Y$ are $v$ and $w$ and consequently the differential of generators not involving them are described by formula $(\ref{Differential})$. The differential of the remaining generators of the model of the product deserves a careful calculation following the pocedure explained in Example \ref{exampleproduct}.
\vskip .2cm
We will write in bold style the terms that appear due to the inclusion of the differentials $\partial c = [b, b]$ and $\partial z=[y, y]$ to simplify the calculations.

First, $s(V_0\otimes W_0)=\langle ax, ay, bx, by \rangle$ and we have 
$$D(ax)=[a, x],\ D(ay)=[a, y],\ D(bx)=[b, x],\ D(by)=[b, y].$$

Next, $s(V_0\otimes W_1)=\langle az, bz \rangle$ and following Definition \ref{L1} we have
\begin{align*}
D(az)&=[a, z]-(-1)^{|a|}\sigma_a (\partial z)=[a, z] \pmb{-2[ay, y]}\\
D(bz)&=[b, z]-(-1)^{|b|}\sigma_b (\partial z)=[b, z] \pmb{+2[by, y]},
\end{align*} 
and $s(V_1\otimes W_0)=\langle cx, cy \rangle$, and therefore
\begin{align*}
D(cx)&=[c, x]-(-1)^{(|c|+1)|x|}\sigma_x (\partial c)=[c, x] \pmb{-2[bx, b]}\\
D(cy)&=[c, y]-(-1)^{(|c|+1)|y|}\sigma_y (\partial c)=[c, y] \pmb{+2[by, b]}.
\end{align*} 

We use the full formula $(\ref{Differential})$ for the generator in $s(V_1\otimes W_1) =\langle cz \rangle$ obtaining
\begin{align*}
D(cz)&=[c, z]-(-1)^{(|c|+1)|z|}\sigma_z(\partial c)-(-1)^{|c|}\sigma_c (\partial z)+(-1)^{|c|}\overline{\partial c}\star \partial z\\
&=[c, z]\pmb{+ 2[bz, b]+2[cy, y]+2[by, by]}.  
\end{align*}

The generators of $s(V_0\otimes W_2) = \langle aw, bw \rangle$ are in principle out of the scope of formula $(\ref{Differential})$, but actually they are covered by Lupton-Smith formula $(\ref{diferentialLuptonSmith})$:
\begin{align*}
D(aw)&=[a, w]-(-1)^{|a|}\sigma_a (\partial w)\\
&=[a, w]-[ax, [y, z]]-[x, [ay, z]]+[x, [y, az]],\\
D(bw)&=[b, w]-(-1)^{|b|}\sigma_b (\partial w)\\
&=[b, w]+[bx, [y, z]]+[x, [by, z]]+[x, [y, bz]].
\end{align*}

Analogously, $s(V_2\otimes W_0)=\langle vx, vy \rangle$ has differentials:
\begin{align*}
D(vx)&=[v, x]-(-1)^{(|v|+1)|x|}\sigma_x (\partial v)\\
&[v, x]-[ax, [b,c]]-[a, [bx, c]]-[a, [b, cx]],\\
D(vy)&=[v, y]-(-1)^{(|v|+1)|y|}\sigma_y (\partial v)\\
&=[v, y]-[ay, [b, c]]+[a, [by, c]]+[a, [b, cy]].
\end{align*}

Next step is $s(V_2\otimes W_1) =\langle cw \rangle$. Although formula $(\ref{Differential})$ is not in principle designed for generators of this type, we can verify that it works correctly. Indeed,
\begin{align*}
D(cw)&=[c, w]-(-1)^{(|c|+1)|w|}\sigma_w (\partial c) - (-1)^{|c|}\sigma_c (\partial w)+(-1)^{|c|}\overline{\partial c}\star\partial w\\
&=[c, w] + \pmb{2[bw, b]}\\
&+[cx, [y, z]]+[x, [cy, z]]+[x, [y, cz]]\\
&\pmb{+2[bx, [by, z]]+2[bx, [y, bz]]+2[x, [by, bz]]},
\end{align*}
and a simple inspection shows that $D^2 (cw)=0$.

Analogously, $s(V_2\otimes W_1)=\langle vz \rangle$, and
\begin{align*}
D(vz)&=[v, z]-(-1)^{(|v|+1)|z|}\sigma_z (\partial v) - (-1)^{|v|}\sigma_v (\partial z)+(-1)^{|v|}\overline{\partial v}\star \partial z\\
&=[v, z] -[az, [b, c]]+[a, [bz, c]]+[a, [b, cz]]\\
&\pmb{ +[vy, y]-2[ay, [by, c]]-2[ay, [b, cy]]+2[a, [by, cy]]},
\end{align*}
verifies that $D^2(vz)=0$.
Finally, $s(V_2\otimes W_2)=\langle vw \rangle$ and formula $(\ref{Differential})$ does not work entirely. Indeed, if we define
\begin{align*}
D(vw)&=[v, w]-(-1)^{(|v|+1)|w|}\sigma_w (\partial v) - (-1)^{|v|}\sigma_v (\partial w)+(-1)^{|v|}\overline{\partial v}\star\partial w\\
&=[v, w]\\
&-\Bigl[aw, [b, \pmb{c}]\Bigr]+\Bigl[a, [bw, \pmb{c}]\Bigr]+\Bigl[a, [b, \pmb{cw}]\Bigr]\\
&+\Bigl[vx, [y, \pmb{z}]\Bigl]+\Bigl[x, [vy, \pmb{z}]\Bigr]+\Bigl[x, [y, \pmb{vz}]\Bigr]\\
&+\Bigl[ [bx, \pmb{c}],[ay, \pmb{z}]\Bigr] -\Bigl[ [bx, \pmb{c}], [y, \pmb{az}]\Bigr]-\Bigl[[b, \pmb{cx}], [ay, \pmb{z}]\Bigr]\\
&+\Bigl[[b, \pmb{cx}], [y, \pmb{az}]\Bigr]+\Bigl[ ax, [[by, \pmb{c}], \pmb{z}]\Bigr]+\Bigl[x, [[by , \pmb{c}], \pmb{az}]\Bigr]\\
&+\Bigl[ax, [[b , \pmb{cy}], \pmb{z}]\Bigr]+\Bigl[ x, [[b , \pmb{cy} ], \pmb{az}]\Bigr] \hskip.08cm -\hskip .08cm\Bigr[ax, [y,[\pmb{bz}, \pmb{c}]]\Bigr]\\
&\hskip .1cm -\hskip .05cm\Bigr[x, [ay, [\pmb{bz}, \pmb{c}]]\Bigr]-\Bigl[ax, [y, [b , \pmb{cz}]]\Bigr] -\Bigl[x, [ay, [b, \pmb{cz}]]\Bigr]\\
&+\Bigl[ a, [\pmb{cx}, [by, \pmb{z}]]\Bigr] +\Bigl[ a, [\pmb{cx}, [y, \pmb{bz}]]\Bigr]+\Bigl[a, [bx, [\pmb{cy}, \pmb{z}]]\Bigr]\\
& \hskip .1cm +\hskip .05cm \Bigr[a, [x, [\pmb{cy}, \pmb{bz}]]\Bigr] + \Bigl[a, [bx, [y, \pmb{cz}]]\Bigr] + \Bigl[ a, [x, [by, \pmb{cz}]]\Bigr]
\end{align*}
Then, taking into account the differentials of generators in $s(V\otimes W)$ including $c$ and/or $z$ we obtain $D^2(vw)=0 + (\text{extra terms})\not=0$.

It can be checked by a long and tedious computation, that we have to add the terms 
$$2\Bigl[a, [bx, [by, bz]]]\Bigr]-2\Bigl[x, [ay, [by, cy]]\Bigr]-2\Bigl[ax, [y, [by, cy]]\Bigr]$$
to the expression given for $D(vw)$ in order to cancel the extra terms and obtain $D^2 (vw)=0$.

As we see, for $3$-cones, terms with three factors in $s(V\otimes W)$ can appear, but an explicit formula for this case is unknown to us.
\end{example}

\noindent\textbf{Acknowledgements.}  
The first author was partially supported by the MICINN grant PID2020-118753GB-I00 of the Spanish
Government and the Junta de Andaluc\'ia grant ProyExcel-00827

The third author was partially supported by Portuguese Funds through FCT (Funda\c c\~ao para a Ci\^encia e a Tecnologia) within the Projects UIDB/00013/2020 and UIDP/00013/2020

This work has been supported by the National Science Centre grant \\ 2016/21/P/ST1/03460 within the European Union's Horizon 2020 research and innovation programme under the Marie Sk\l{}odowska-Curie grant agreement No. 665778.

\begin{flushright}
	\includegraphics[width=38px]{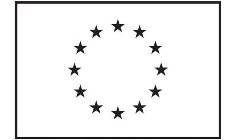}%
\end{flushright}

\end{document}